%% file: least_no_convex0122.tex
\documentclass[12pt]{amsart}
\usepackage{verbatim, latexsym, amssymb, amsmath,color}
\usepackage{amsfonts}
\usepackage{amsmath}
\usepackage{graphicx}
\usepackage{color}
\usepackage{stmaryrd}
\usepackage{latexsym,amscd,amsthm,amsxtra}

\usepackage{amsfonts}

\usepackage{mathrsfs}
\usepackage{latexsym,amssymb}

\usepackage{epsfig,color}
\usepackage{blindtext}
\usepackage{enumerate}
\usepackage{hyperref}
\usepackage{url}
\usepackage{bbm}
\usepackage{filecontents}
\usepackage{nicefrac,mathtools}
\DeclareGraphicsExtensions{.pdf,.jpeg,.png}
\usepackage{epstopdf}
\usepackage{cancel} 
\usepackage[normalem]{ulem} 

\usepackage{color}
\usepackage[msc-links, lite]{amsrefs}
\usepackage{geometry}
\newtheorem{theorem}{Theorem}[section]

\newtheorem{proposition}[theorem]{Proposition}
\newtheorem{lemma}[theorem]{Lemma}
\newtheorem{corollary}[theorem]{Corollary}

\newtheorem{question}[theorem]{Question}

\newtheorem{claim}[]{Claim}

\theoremstyle{definition}
\newtheorem{definition}[theorem]{Definition}
\theoremstyle{remark}
\newtheorem{remark}[theorem]{Remark}

\numberwithin{equation}{section}

\newcommand{\mf}{\mathbf}
\newcommand{\mb}{\mathbb}
\newcommand{\mc}{\mathcal}

\newcommand{\wti}{\widetilde}

\newcommand{\M}{\mathbf M}
\newcommand{\m}{\mathbf m}

\newcommand{\Z}{\mathcal Z}

\newcommand{\R}{\mathbb{R}}

\newcommand{\n}{\mathbf n}

\DeclareMathOperator{\area}{Area}

\DeclareMathOperator{\Ric}{Ric}

\DeclareMathOperator{\spt}{spt}

\title{Free boundary minimal hypersurfaces with least area}

\date{\today}

\author{Qiang Guang}
\address{Department of Mathematics, University of California Santa Barbara, Santa Barbara, CA 93106, USA}
\email{guang@math.ucsb.edu}

\author{Zhichao Wang}
\address{School of Mathematical Sciences, Peking University Yiheyuan Road 5, Beijing, P.R.China, 100871}
\email{wangzhichaonk@gmail.com}

\author{Xin Zhou}
\address{Department of Mathematics, University of California Santa Barbara, Santa Barbara, CA 93106, USA}
\email{zhou@math.ucsb.edu}

\begin{document}

\begin{abstract}
In this paper, we prove the existence of the free boundary minimal hypersurface of least area in compact manifolds with boundary. Such hypersurface can be viewed as the ground state of the volume spectrum introduced by Gromov. Moreover, we characterize the orientation and Morse index of them.
\end{abstract}

\maketitle

\section{Introduction}

Let $(M^{n+1},\partial M)$ be a compact Riemannian manifold with boundary. We say that a smooth embedded hypersurface $\Sigma^n\subset M^{n+1}$ is a {\em free boundary minimal hypersurface}, abbreviated as FBMH, if $\Sigma$ has vanishing mean curvature and $\partial \Sigma$ meets $\partial M$ orthogonally. FBMHs arise variationally as critical points of the area functional among all hypersurfaces in $M$ with boundary constrained freely on $\partial M$. The investigation of such hypersurfaces dates back at least to Courant \cite{Cou40} and Lewy \cite{Le51}, and there were intense study afterward, e.g., \citelist{\cite{HN79}\cite{MY80}\cite{Str84}\cite{GJ86}\cite{Jost86}\cite{Ye91}}. Many new progress were made in recent years. Among them, Schoen-Fraser \citelist{\cite{FS11}\cite{FS16}} constructed many examples of free boundary minimal surfaces in the round ball and found a deep relation of them with the extremal eigenvalue problem. More examples of minimal surfaces with free boundary in the round three-ball were recently found by Folha-Pacard-Zolotareva \cite{FPZ17}, Ketover \cite{Ket16b} and Kapouleas-Li \cite{KL17}. Maximo-Nunes-Smith \cite{MNS17} constructed an annuli type of such minimal surfaces in certain convex three-manfolds using degree theory.  Lastly and foremost, global variational theory for constructing FBMHs in an arbitrary manifold via min-max method was initiated by Almgren \citelist{\cite{Alm62}\cite{Alm65}}, and completely established by the last author with Li \cite{LZ16}; see \citelist{\cite{GJ86}\cite{Jost86}\cite{Li15}\cite{DeRa16}} for partial results and \cites{Str84, Fra00} for the mapping approach. Inspired by Almgren's pioneer work, Gromov introduced the non-linear spectrum of area functional \cite{Gr03}, for which a Weyl law was established by Liokumovich-Marques-Neves \cite{LMN16}. In this perspective, FBMHs are eigenstates of area functional. Compared with the classical spectrum theory, one natural question is whether the ground state among all area spectrum exists. In particular,

\begin{question}
\label{Q1}
Does there exist a FBMH whose area is less than all others?
\end{question}

In this paper, we give a complete affirmative answer to above question. In particular, we prove the existence of smooth embedded free boundary minimal hypersurfaces in $(M^{n+1},\partial M,g)$ $(2\leq n\leq 6)$ which minimize the area among all such hypersurfaces. To be precise, let $\mc{O}$ be the collection of all embedded compact orientable  FBMHs in $M$ and $\mc{U}$ be the collection of all non-orientable ones. Set
\[\mc{A}_1(M,\partial M)=\inf\big(\{\area(\Sigma),\,\Sigma\in \mc{O}  \}\cup \{2\area(\Sigma),\,\Sigma\in \mc{U} \}\big).\]

Our main result is the following:
\begin{theorem}\label{thm:main}
Let $M^{n+1}$ be a smooth compact orientable Riemannian manifold with  boundary $\partial M$ and $2\leq n\leq 6$. Then there exists a smooth embedded free boundary minimal hypersurface $\Sigma$ in $M$ such that $\mc{A}_1(M,\partial M)$ is realized by $\Sigma$. Moreover, $\Sigma$ satisfies the following:
\begin{enumerate}
\item If $\Sigma\in \mc{O}$, then $\Sigma$ has index $0$ or $1$. In the latter case, $M$ is the min-max minimal hypersurface corresponding to the fundamental class.
\item If $\Sigma\in \mc{U}$, then $\Sigma$ is stable. Moreover, the 2-sheeted covering of $\Sigma$ is stable.
\end{enumerate}
\end{theorem}

Note that the area of non-orientable hypersurfaces is counted with multiplicity two, and the reason is that non-orientable minimal hypersurfaces produced by the min-max method have even multiplicity by \cites{Zhou15, Wang}. $\mc{A}_1(M,\partial M)$ will be called the {\em least area among embedded free boundary minimal hypersurfaces in $M$}.

\begin{remark}
One of the main features of our result is that we allow FBMHs to be improper, i.e., the interior of FBMHs may touch the boundary of the ambient manifold (see Definition \ref{def:touching}). Note that we do not assume any boundary convexity on the ambient manifold, so this touching phenomenon is allowed to happen.
\end{remark}

\vspace{1em}
For closed manifolds, the least area closed minimal hypersurfaces and min-max hypersurfaces have been well studied. For instance, in three dimension, the least area Heegaard minimal surface always exists by a classical compactness theorem, and Marques-Neves \cite{MN12} proved that it is produced by min-max method in any closed three-manifold which admits no stable minimal surfaces. In higher dimensions, a priori the existence of least area minimal hypersurfaces was not known due to the lack of compactness result; nevertheless, as a byproduct of the study of the Morse index problem, the last author \cites{Zhou15, Zhou17} proved that the min-max hypersurface has least area in manifolds with positive Ricci curvature.
Later on, without assuming any curvature conditions, Mazet-Rosenberg \cite{MaRo17} further proved that the least area is achieved either by a stable closed minimal hypersurface or by a min-max hypersurface of Morse index one in any closed $(n+1)$-manifold ($2\leq n\leq 6$).

For compact manifolds with non-empty boundary, the first result toward Question \ref{Q1} was obtained by the second author \cite{Wang} for manifolds with non-negative Ricci curvature and convex boundary. In particular, it was proven that the min-max FBMH is orientable of multiplicity one, with Morse index one. Furthermore, the min-max FBMH has least area among all embedded orientable ones. Our resolution of Question \ref{Q1} can be viewed as a free boundary analog of the result in \cite{MaRo17}. However, without assuming any convexity assumption on $\partial M$, the situation turns to be very subtle, especially due to the touching phenomenon of the min-max minimal hypersurfaces, which was predicted to generally exist in \cite{LZ16}. More precisely, when the boundary $\partial M$ is non-convex, such a FBMH generally can be non-proper, or equivalently, the interior of a FBMH can touch $\partial M$ in a non-empty set.  This brings in essential new challenges for the deformation trick used in \citelist{\cite{MaRo17}\cite{ Wang}}. In course of the proof, we introduce several new ideas to deal with this issue, and we believe that our new technique will also be useful in other problems related to FBMHs.

\begin{remark}
Throughout the paper, $(M^{n+1},\partial M,g)$ is always a manifold with $2\leq n\leq 6$ unless explicitly stated otherwise. Moreover, it can always be embedded to a closed manifold $\widetilde M$ which has the same dimension with $M$.

A FBMH in this paper is always allowed to be closed.
\end{remark}

Our main result follows similar strategy used by Mazet-Rosenberg \cite{MaRo17}. The first step is to consider stable FBMHs. Using the curvature estimates and compactness result in \cite{GLZ16} and \cite{ACS17}, we will show that there is a stable one minimizing the area among all stable ones. If $\mc{A}_1(M,\partial M)$ is equal to the least area of stable ones, then we are done. Otherwise, we proceed to the second step, in which we consider all embedded  orientable and unstable FBMHs with area less than the least area of all stable ones. We will show that each of such hypersurfaces can be embedded into a sweepout of $M$. Then the last step is to apply the min-max theory by Li-Zhou \cite{LZ16} for the sweepout so that we can obtain a free boundary one with least area, which will imply that $\mc{A}_1(M,\partial M)$ is realized. Apparently, the key step is to construct a good sweepout for any orientable and unstable FBMH with area less than that of all stable ones.

To prove the existence of such sweepouts, we use contradiction arguments. Assuming that such a good sweepout does not exist, then a new FBMH is produced by the min-max theory \cite{LZ16}. We emphasis that a generalized type of Almgren-Pitts theory \citelist{\cite{Pit76}\cite{MN14}\cite{LZ16}} is essentially used here, which is different from the continuous min-max theory used in \citelist{\cite{MN12}\cite{Zhou15}\cite{MaRo17}\cite{Wang}}. The necessity of Almgren-Pitts setting is due to that continuous min-max theories \citelist{\cite{Code03}\cite{DeRa16}} are not suitable for compact manifolds without convexity assumptions of boundary.

\vspace{1em}
The paper is organized as follows. In Section \ref{sec:preliminary}, we recall some basic definitions and prove that a FBMH can be perturbed to  a barrier. In Section \ref{sec:stable}, we consider the case when least area is attained by the area of a sequence of stable free boundary ones. In Section \ref{sec:min-max}, we introduce the min-max theory developed by Li-Zhou \cite{LZ16}. Also, we prove that it still works when there is a barrier. In Section \ref{sec:unstable}, we embed each FBMH with area less than $\mathcal A_\mathcal S(M,\partial M)$ (see (\ref{E:As})) into a good sweepout. Finally, in Section \ref{sec:proof}, we give the proof of our main result Theorem \ref{thm:main}.

\vspace{1em}
{\bf Acknowledgement:} 
The authors would like to thank Professor Gang Tian for his interest in this paper.

\section{Preliminaries}\label{sec:preliminary}
In this section, we first recall some basic definitions and preliminary results for FBMHs and then prove that each one with non-zero first eigenvalue for the Jacobi operator can be perturbed to a barrier.

\begin{definition}(\cite{LZ16})\label{def:touching}
Let $M^{n+1}$ be a smooth manifold with boundary $\partial M$ and $\Sigma^n$ a smooth $n$-dimensional manifold with boundary $\partial \Sigma$. We say that  a smooth embedding $\phi: \Sigma\to M$ is an \emph{almost proper embedding} of $\Sigma$ into $M$ if $\phi(\Sigma)\subset M$ and $\phi(\partial \Sigma)\subset \partial M$. We often write $(\Sigma,\partial \Sigma)\to (M,\partial M)$ and say that $\Sigma$ is an {\em almost properly embedded hypersurface} in $M$.
\end{definition}

For an almost properly embedded hypersurface $(\Sigma, \partial\Sigma)$, we allow the interior of $\Sigma$ to touch $\partial M$. That is to say: $\mathrm{int}(\Sigma)\cap \partial M$ may be non-empty.  We usually call $\mathrm{int}(\Sigma)\cap \partial M$ the {\em touching set} of $\Sigma$.

\subsection{The Morse Index}
Let $(M^{n+1},\partial M,g)$ be a compact Riemannian manifold with boundary. Suppose that $(\Sigma^n,\partial\Sigma)\subset (M^{n+1},\partial M)$ is a smooth embedded FBMH. The quadratic form associated to the second variation formula is defined as
\begin{equation}\label{equ:index_form}
Q(v,v)=\int_\Sigma \left(|\nabla^\perp v|^2- \Ric_M(v,v)-|A|^2|v|^2 \right)\,d\mu  - \int_{\partial \Sigma} h^{\partial M}(v,v)\,ds,
\end{equation}
where $v$ is a section of the normal bundle of $\Sigma$, $A$ and $h^{\partial M}$ are the second fundamental forms of $\Sigma$ and $\partial M$, respectively.

The \emph{Morse index} of $\Sigma$ is defined as the number of negative eigenvalues of the quadratic form $Q$, and  a FBMH is called stable if $Q$ is non-negative.

If $\Sigma$ is two-sided, i.e., there exists a globally defined unit normal vector field $\mf{n}$ on $\Sigma$, any normal vector field on $\Sigma$ has the form $\phi \mf{n}$, where $\phi \in C^{\infty}(\Sigma)$. Then the quadratic form can be expressed as
\[Q(\phi,\phi)=\int_\Sigma \left(|\nabla_\Sigma \phi|^2- (\Ric_M(\mf{n},\mf{n})+|A|^2)\phi^2 \right)\,d\mu  - \int_{\partial \Sigma} h^{\partial M}(\mf{n},\mf{n})\phi^2\,ds.\]
Integration by parts gives that
\begin{equation}
Q(\phi,\phi)=-\int_\Sigma \phi L\phi\,d\mu+\int_{\partial \Sigma} \left(\phi \frac{\partial \phi}{\partial \eta}-h^{\partial M}(\mf{n},\mf{n})\phi^2\right)\,ds,
\end{equation}
where $\eta$ is the outward unit co-normal of $\partial \Sigma$ and  $L$ is the Jacobi operator of $\Sigma$
\begin{equation}
\label{E:Jacobi operator}
L=\Delta_\Sigma +|A|^2+\Ric_M(\mf{n},\mf{n}).
\end{equation}
Let $\lambda_1<\lambda_2\leq\lambda_3\leq...$ be eigenvalues of the following system
\begin{equation}
\label{neumann problem}
\left\{
\begin{array}{ll}
Lu+\lambda u=0 \quad & \text{on }  \Sigma,\\
\frac{\partial u}{\partial \eta}=h^{\partial M}(\mf{n},\mf{n})u \quad & \text{on } \partial \Sigma.
\end{array}
\right.
\end{equation}
The index of $\Sigma$ is just equal to the number of negative eigenvalues of (\ref{neumann problem}).

\subsection{Construction of Barriers}We now prove that certain FBMH can be perturbed to a barrier (in the sense described as follows). Let us first introduce some notations.
\begin{definition}
\label{D:manifold with boundary and portion}
For a manifold with piecewise smooth boundary, $N$ is called a \emph{manifold with boundary $\partial N$ and portion $T$} if
\begin{itemize}
  \item $\partial N$ and $T$ are smooth, which may be disconnected;
  \item $\partial N\cup T$ is the boundary of $N$.
\end{itemize}
We will denote it by $(N,\partial N,T)$ (see Figure I).

Moreover, $T$ is called a \emph{barrier} if $T$ is mean convex and $\langle\nu_{\partial N},\n_T\rangle<0$ on $T\cap \partial N$, where $\nu_{\partial N}$ and $\n_T$ are the outward pointing normal vector fields of $\partial N$ and $T$.
\end{definition}


\begin{figure}[h]
\begin{center}
\def\svgwidth{0.4\columnwidth}
  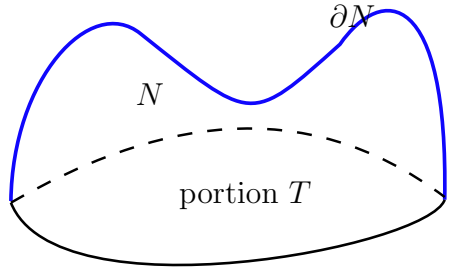
  \label{graph:portion}
  \caption{Manifold with boundary $\partial N$ and portion $T$.}
\end{center}
\end{figure}

\begin{definition}\label{def:available vector}
 Let $\Sigma\subset M$ be a FBMH which may have touching set (see Figure II). Denote $\mathfrak X(M,\Sigma)$ as the collection of vector fields $X\in \mathfrak X(\widetilde M)$ in $\widetilde{M}$ such that $X(p)\in T_p(\partial M)$ for $p$ in a neighborhood of the boundary of $\Sigma$.
\end{definition}
\begin{remark}
Note that in the above definition, a vector field $X\in \mathfrak X(\widetilde M)$ along the touching set $\mathrm{int}(\Sigma)\cap \partial M$ may not be tangential to $\partial M$. In fact, $X$ is allowed to point inward or outward of $M$.
\end{remark}
\begin{figure}[h]
\begin{center}
\def\svgwidth{0.4\columnwidth}
  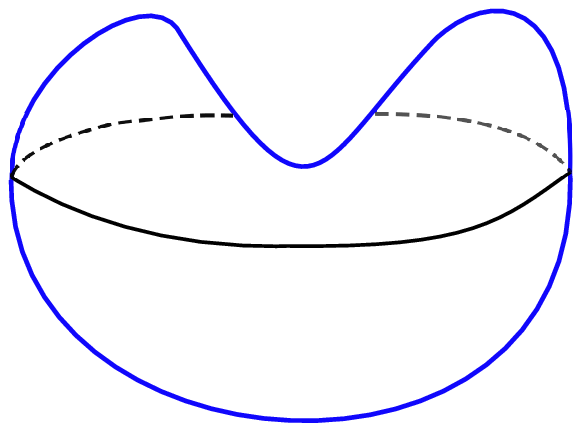
  \label{graph:touching}
  \caption{}
\end{center}
\end{figure}

Let $(F_s)_{0\leq s\leq 1}$ be a family of diffeomorphisms of $\widetilde M$ corresponding to $X\in\mathfrak X(M,\Sigma)$. Note that $F_s$ is allowed to push $\Sigma$ out of $\widetilde M$ along the touching set $\mathrm{int}(\Sigma)\cap \partial M$. For any $p\in\Sigma$, there is a coordinate chart $(x^1, x^2, \cdots, x^n, s)$ around $p$ such that
\begin{enumerate}
\item $(x, s) \in\Sigma$ iff $s=0$;
\item $\{e_i=\frac{\partial}{\partial x^i}\}$ is an orthogonal basis at $p$;
\item $\partial_s=X$.
\end{enumerate}

Let $(h_{ij})$ be the second fundamental form under the coordinates $(x^1, x^2, \cdots, x^n)$. Let $\n$ and $\n_s$ be the unit normal vector fields of $\Sigma$ and $F_s(\Sigma)$ respectively. Set
\[u=\langle \n_s,\nu_{\partial M}\rangle,\]
which is a smooth function on a neighborhood of $\partial\Sigma$ in $\partial M$.

\begin{lemma}\label{lem:derivative}
If in addition that $\Sigma$ is orientable and  $X|_{\Sigma}=f\n$ for some smooth function $f$, we have the following equation at $p$,
\begin{gather}
\nabla_Xe_i\big|_{s=0}=-fh_{ij}e_j+f_i\n.\\
\nabla_X\n_s\big|_{s=0}=-\nabla f,\label{equ:normal}\\
\nabla_Xu\big|_{\partial\Sigma}=fh^{\partial M}(\n,\n)-\langle\nabla f,\nu_{\partial M}\rangle, \label{equ:acute angle}\\
\nabla_XH\big|_{\Sigma}=Lf,\label{equ:positive mean curvature}
\end{gather}
where $L$ is the Jacobi operator (\ref{E:Jacobi operator}).
\end{lemma}
\begin{proof}
By definition,
\[\partial_s=X.\]
Hence,
\[\nabla_Xe_i\big|_{s=0}=\nabla_{e_i}X\big|_{s=0}=\nabla_{e_i}(f\n).\]
We conclude that
\[\nabla_X e_i\big|_{s=0}=-fh_{ij}e_j+f_i\n.\]
Then for any $i$,
\[
    \langle\nabla_X\n_s,e_i\rangle\big|_{s=0}=-\langle\n,\nabla_X{e_i}\rangle=\langle-\nabla f,e_i\rangle,
\]
and this proves (\ref{equ:normal}). To confirm (\ref{equ:acute angle}),
\begin{align*}
    \nabla_Xu\Big|_{\partial\Sigma}&=\nabla_X\langle\n_s,\nu_{\partial M}\rangle\Big|_{\partial\Sigma}\\
                                   &=\langle\nabla_X\n_s,\nu_{\partial M}\rangle\Big|_{s=0}+\langle\n_s,\nabla_{X}\nu_{\partial M}\rangle\Big|_{s=0}\\
                                   &=-\langle\nabla f,\nu_{\partial M}\rangle+fh^{\partial M}(\n,\n).
\end{align*}
The last can be derived by a standard computation.
\end{proof}

By Lemma \ref{lem:derivative}, we conclude that
\begin{proposition}
\label{prop:barrier:unstable}
Let $(M,\partial M, \Sigma)$ be a compact manifold with portion $\Sigma$, where $\Sigma$ is an unstable FBMH. Then there is a family of hypersurfaces $\{\Sigma_s\}_{0\leq s\leq \tau}$ for small $\tau>0$ such that $(\Sigma_s,\partial\Sigma_s)\subset(M,\partial M)$, and they satisfy
  \begin{enumerate}
    \item\label{item:devide} $\Sigma_0=\Sigma$ and $\Sigma_s\cap\Sigma=\emptyset$ for $s\neq 0$;
    \item\label{item:proper} $\Sigma_s$ is a properly embedded hypersurface with boundary on $\partial M$ for each $s\neq 0$;
    \item\label{item:area}   $\area(\Sigma_t)$ is decreasing with respect to $t$.
  \end{enumerate}
Moreover, if we set
\[ M_s:=M\setminus(\cup_{0\leq t<s}\Sigma_t),\]
$(M_s,\partial M_s,\Sigma_s)$ is a compact manifold with portion $\Sigma_s$, and $\Sigma_s$ is also a barrier.
\end{proposition}

\begin{proof}
Since $\lambda_1(\Sigma)<0$, there is a positive function $f$ satisfying
\begin{equation}\label{equ:barrier function}
\left\{\begin{array}{ll}
        Lf>0 \ & \text{in\ } \Sigma\\
        \frac{\partial f}{\partial\eta}<h^{\partial M}(\n,\n)f\ & \text{on\ } \partial\Sigma.
        \end{array}\right.
\end{equation}

To see this, we define a perturbed quadratic form
\[Q_t(\phi,\phi)=-\int_\Sigma \phi (L-t)\phi\,d\mu+\int_{\partial \Sigma} \left(\phi \frac{\partial \phi}{\partial \eta}-\left(h^{\partial M}(\mf{n},\mf{n}) -t\right)\phi^2\right)\,ds,\]
for $t\in\mathbb R$.
Let $\phi_1$ be the first eigenfunction of (\ref{neumann problem}). Due to the fact that
\[ \lim_{t\rightarrow 0}Q_t(\phi_1,\phi_1)=Q(\phi_1,\phi_1)<0, \]
there is some constant $\epsilon>0$ small enough such that the first eigenvalue of $Q_\epsilon$ is negative. Let $f$ be the first eigenfunction of $Q_\epsilon$. Then $f>0$ and
\begin{equation*}
\left\{\begin{array}{ll}
        Lf-\epsilon f>0 \ & \text{in\ } \Sigma\\
        \frac{\partial f}{\partial\eta}=\left(h^{\partial M}(\n,\n)-\epsilon\right)f\ & \text{on\ } \partial\Sigma.
        \end{array}\right.
\end{equation*}
So $f$ satisfies (\ref{equ:barrier function}).

Notice that $(M,\partial M)$ can always be isometrically embedded into some closed manifold $\widetilde M$ with $\mathrm{dim}\widetilde M=\mathrm{dim} M$. Take $X\in\mathfrak X(M,\Sigma)$ such that $X\big|_{\Sigma}=-f\n$. Denote $(F_s)_{-\tau\leq s\leq\tau}$ as the family of diffeomorphisms of $\widetilde M$ corresponding to $X$. Set
\[\Sigma_s=F_s(\Sigma)\cap M.\]
Let us verify that these hypersufaces satisfy all requirements. (\ref{item:devide}) holds since $f$ is positive. (\ref{item:proper}) will be satisfied by choosing suitable $f$ and small $\tau$ even if $\Sigma$ touches $\partial M$ (see Figure III). It follows from (\ref{equ:acute angle})  and  (\ref{equ:positive mean curvature}) that $\Sigma_s$ is a barrier. By the second variation formula and shrinking $\epsilon$ smaller if needed, the area functional is decreasing.
\end{proof}

\begin{figure}[h]
\begin{center}
\def\svgwidth{0.4\columnwidth}
  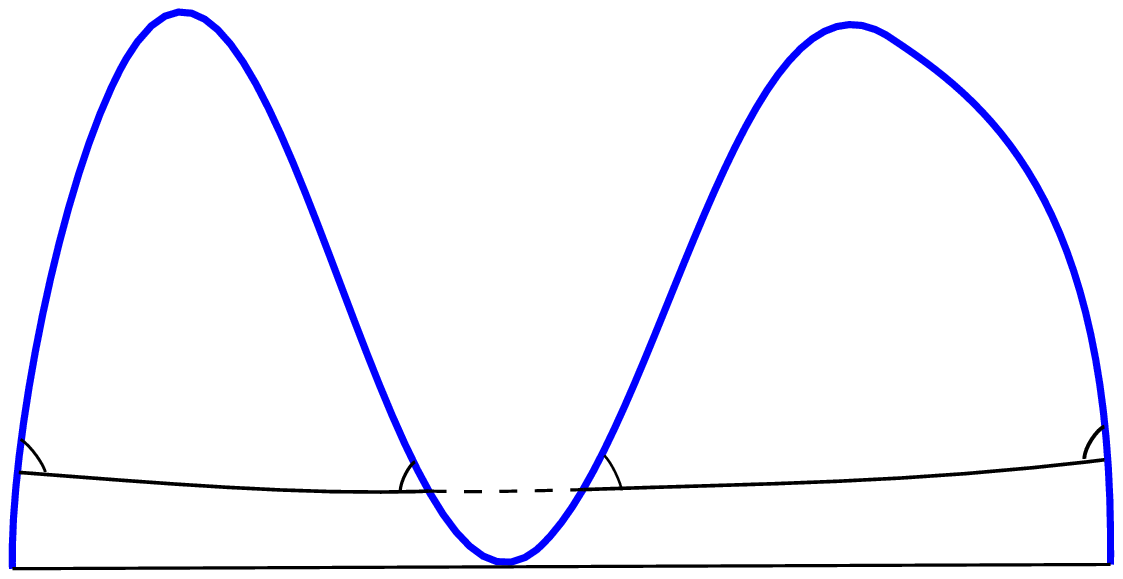
  \label{graph:barrier}
  \caption{}
\end{center}
\end{figure}

If $\Sigma$ is stable, there are no barriers inside $M$. However, we can construct a barrier outside $M$. Let $(M,\partial M,\Sigma)$ be a compact manifold with portion $\Sigma$. Let $\widetilde M$ be a closed manifold with $\mathrm{dim}\widetilde M=\mathrm{dim}M$ such that $M$ can be isometrically embedded into $\widetilde M$.
\begin{proposition}\label{prop:barrier:stable}
Suppose that $\lambda_1(\Sigma)>0$. Then there exists a family of hypersurface $\{\Sigma_t\}_{0\leq t\leq \tau}$ in $\widetilde M$ for $\tau$ small enough such that
\begin{enumerate}
  \item \label{item:devide:stable} $\Sigma_0=\Sigma$ and $\Sigma_s\cap M=\emptyset$ for $s>0$;
  \item \label{item:proper:stable}each $\Sigma_t$ is a properly embedded hypersurface with boundary;
  \item \label{item:area:stable} $\lim_{t\rightarrow 0}\area(\Sigma_t)=\area(\Sigma)$.
\end{enumerate}
If setting
\[ M_s:=\cup_{0\leq t\leq s}\Sigma_t\cup M,\]
we also have that $(M_s,\partial M_s,\Sigma_s)$ is a compact manifold with portion $\Sigma_s$, where $\Sigma_s$ is also a barrier.
\end{proposition}
\begin{proof}
Since $\lambda_1(\Sigma)>0$, there is a positive function $f$ satisfying
\begin{equation}\label{equ:barrier function:stable}
\left\{\begin{array}{ll}
        Lf<0 \ & \text{in\ } \Sigma\\
        \frac{\partial f}{\partial\eta}>h^{\partial M}(\n,\n)f\ & \text{on\ } \partial\Sigma.
        \end{array}\right.
\end{equation}

Let $\phi_1$ be the first eigenfunction of (\ref{neumann problem}). Due to the fact that
\[ \lim_{t\rightarrow 0}Q_t(\phi_1,\phi_1)=Q(\phi_1,\phi_1)>0, \]
there exists $\epsilon>0$ small enough such that the first eigenvalue of $Q_{-\epsilon}$ is positive. Let $f$ be the first eigenfunction of $Q_\epsilon$. Then $f>0$ and
\begin{equation*}
\left\{\begin{array}{ll}
        Lf+\epsilon f<0 \ & \text{in\ } \Sigma\\
        \frac{\partial f}{\partial\eta}=\left(h^{\partial M}(\n,\n)+\epsilon\right)f\ & \text{on\ } \partial\Sigma.
        \end{array}\right.
\end{equation*}
So $f$ satisfies (\ref{equ:barrier function:stable}).

Take $X\in\mathfrak X(M,\Sigma)$ such that $X\big|_{\Sigma}=f\n$. Denote $(F_s)_{0\leq s\leq\tau}$ as the family of diffeomorphisms of $\widetilde M$ corresponding to $X$. Set
\[\Sigma_s=F_s(\Sigma).\]
Let us verify that these hypersufaces satisfy all requirements. (\ref{item:devide:stable}) holds since $f$ is positive. (\ref{item:proper:stable}) will be satisfied by choosing suitable $f$ and small $\tau$ even if $\Sigma$ touches $\partial M$. It follows from  (\ref{equ:acute angle}) and (\ref{equ:positive mean curvature}) that $\Sigma_s$ is a barrier (see Figure IV).
\end{proof}

\begin{figure}[h]
\begin{center}
\def\svgwidth{0.4\columnwidth}
  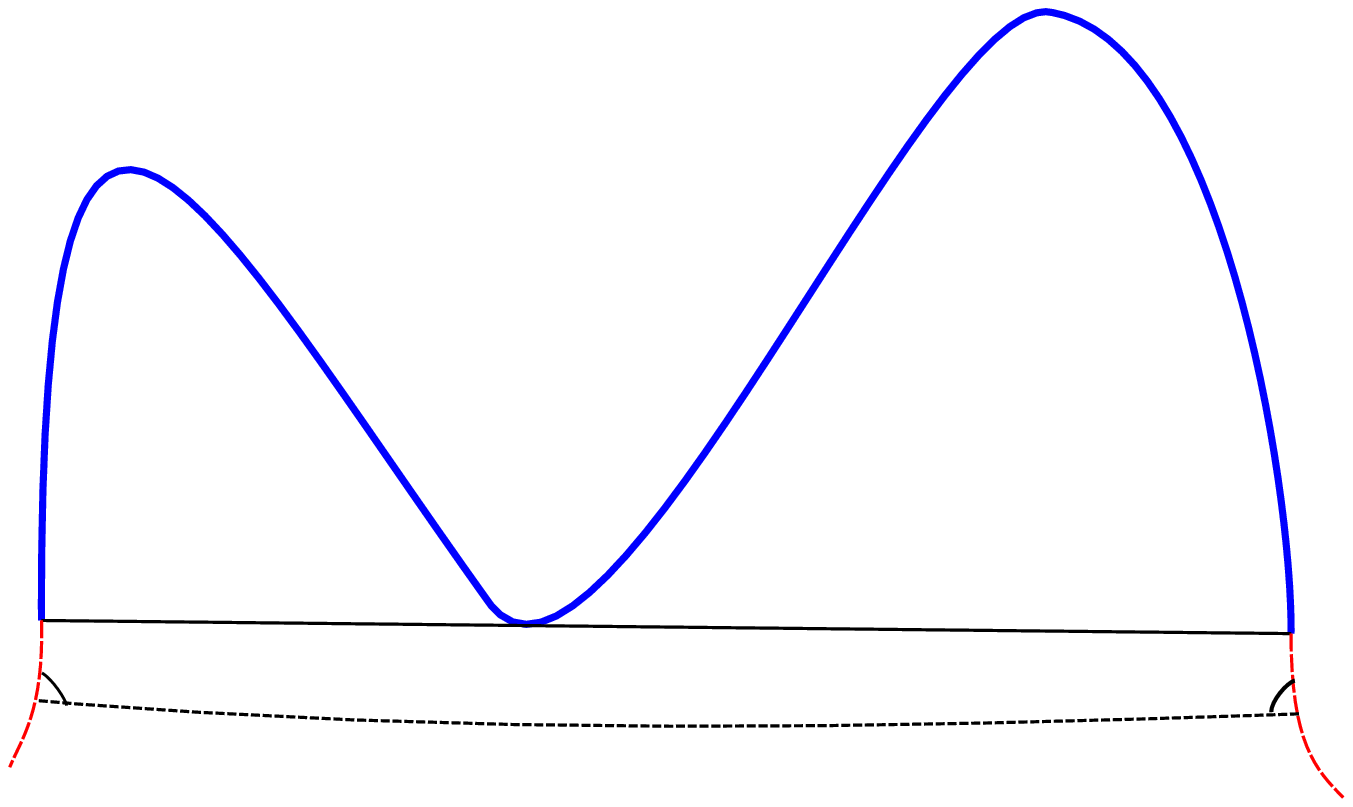
  \label{graph:barrier_stable}
  \caption{}
\end{center}
\end{figure}

\subsection{Bumpy Metric Theorem}\label{thm:bumpy}
For any smooth manifold $N$, White (\citelist{\cite{Whi91}\cite{Whi17}}) proved that a generic $C^k$ ($k\geq 3$ or $k=\infty$) metric on $N$ is \emph{bumpy} in the sense that no closed minimal submanifolds of $N$ has a nontrivial Jacobi field. In this paper, we need the following version of the bumpy metric theorem for FBMHs which is essentially due to Ambrozio-Carlotto-Sharp \cite{ACS17}. 
\begin{theorem}(\cite{ACS17}*{Theorem 9})
Let $\wti{M}^{n+1}$ be a smooth closed manifold and $N^n$ be a smooth embedded closed hypersurface in $\wti{M}$.  Suppose that $k$ is an integer $\geq 3$ or that $k=\infty$. Then a generic $C^k$ Riemannian metric on $\wti{M}$ is bumpy in the following sense: if $\Sigma^n$ is an embedded free boundary minimal hypersurface in $\wti{M}$ with free boundary lying on $N$, then $\Sigma$ or its finite-sheeted covering has no non-trivial Jacobi fields.
\end{theorem}
\begin{remark}
In this theorem, we allow $\Sigma$ to penetrate the constraint hypersurface $N$. In \cite{ACS17}, they only stated the result for $\Sigma$ which does not penetrate $N$, or equivalently $\Sigma\cap N=\partial\Sigma$. Nevertheless, the result also holds true in our situation since the proof in \cite{ACS17} identifies a tubular neighborhood of $\Sigma$ with that of the zero section in the normal bundle of $\Sigma$. Under this identification, the analytic arguments are exactly the same.
\end{remark}

\section{Stable free boundary minimal hypersurfaces}\label{sec:stable}
In this section, we consider stable FBMHs and we will show that there exists one which minimizes the area among all stable ones.

Let $M^{n+1}$ be a compact manifold with boundary $\partial M$ and $2\leq n\leq 6$. Let $\mc{O}_{\mathcal S}$ be the collection of all embedded compact orientable stable FBMHs in $M$ and $\mc{U}_{\mathcal S}$ be the collection of all non-orientable stable ones.  We define
\begin{equation}
\label{E:As}
\mc{A}_{\mathcal S}(M,\partial M)=\inf\big(\{\area(\Sigma),\,\Sigma\in \mc{O}_{\mathcal S}  \}\cup \{2\area(\Sigma),\,\Sigma\in \mc{U}_{\mathcal S} \}\big).
\end{equation}

If there are no such embedded stable hypersurfaces in $M$, then we will set $\mc{A}_{\mathcal S}(M,\partial M)=+\infty$.

We need the following compactness result (see \cites{GLZ16, ACS17}).
\begin{theorem}[\citelist{\cite{GLZ16}\cite{ACS17}}]
\label{ques:stable:compt:multi}
Let $\{\Sigma_k\}$ be a sequence of smooth  embedded stable free boundary minimal hypersurfaces in $M$ with $\sup_k\area(\Sigma_k)<+\infty$. Then up to a subsequence, $\Sigma_k$ converges smoothly and locally uniformly with multiplicity $m\in \mb{N}$ to a smooth  embedded stable free boundary minimal hypersurface $\Sigma$ in $M$. Moreover,
\begin{enumerate}
\item If $\Sigma$ is two-sided, then the multiplicity $m=1$ and $\Sigma_k$ is diffeomorphic to $\Sigma$ eventually.
\item If $\Sigma$ is one-sided, then either $m=1$ and $\Sigma_k$ is eventually diffeomorphic to $\Sigma$  or $m=2$ and $\Sigma_k$ is eventually diffeomorphic to the two-sided covering of $\Sigma$.
\end{enumerate}
\end{theorem}

Next, using Theorem \ref{ques:stable:compt:multi}, we show that $\mc{A}_{\mathcal S}(M,\partial M)$ is realized if it is finite.

\begin{theorem}
\label{ques:stable:realize}
$\mc{A}_{\mathcal S}(M,\partial M)$ is realized if it is finite: either there exists $\Sigma\in \mc{O}_\mathcal S$ such that $\area(\Sigma)=\mc{A}_\mathcal S(M,\partial M)$ or $\Sigma\in \mc{U}_\mathcal S$ such that $2\area(\Sigma)=\mc{A}_\mathcal S(M,\partial M)$.
\end{theorem}
\begin{proof}
Since $\mc{A}_\mathcal S(M,\partial M)$ is finite, we assume  that there is a sequence $\{\Sigma_k\}_{k\in \mb{N}}$ in $\mc{O}_\mathcal S$ (or in $\mc{U}_\mathcal S$)  such that $\area(\Sigma_k)\to \mc{A}_\mathcal S(M,\partial M)$ (or $2\area(\Sigma_k)\to \mc{A}_\mathcal S(M,\partial M)$).

First, we consider the case when the sequence $\{\Sigma_k\}_{k\in \mb{N}}$ is in $\mc{O}_\mathcal S$. We may just apply Theorem \ref{ques:stable:compt:multi} or we can argue as follows. Since $\{\Sigma_k\}_{k\in \mb{N}}$ is a sequence of smooth embedded stable minimal hypersurfaces with uniform area bound, we have uniform curvature estimates for $\{\Sigma_k\}_{k\in \mb{N}}$ (see \cite{GLZ16}). Then
the compactness result in \cite{GLZ16} (see also \cite{ACS17}) implies that, after passing to a subsequence, $\{\Sigma_k\}_{k\in \mb{N}}$ converges smoothly and locally uniformly (possibly with multiplicity) to a smooth embedded stable FBMH $\Sigma$ in $M$. We consider two scenarios:
\begin{itemize}
  \item If $\Sigma$ is two-sided, then for $k$ sufficiently large, $\Sigma_k$ can be written as an entire graph over $\Sigma$ and the convergence is of multiplicity one. In this case, we have $\area(\Sigma)=\lim_{k\to \infty} \area(\Sigma_k)=\mc{A}_\mathcal S(M,\partial M)$.
  \item  If $\Sigma$ is one-sided, then for $k$ sufficiently large, $\Sigma_k$ is an entire two-sheeted graph over $\Sigma$ and the convergence is of multiplicity two. In this case, we have $2\area(\Sigma)=\lim_{k\to \infty}\area(\Sigma_k)=\mc{A}_\mathcal S(M,\partial M)$.
\end{itemize}

Next, we assume that the sequence $\{\Sigma_k\}_{k\in \mb{N}}$ is in $\mc{U}_\mathcal S$. Again by the compactness result, Theorem \ref{ques:stable:compt:multi}, a subsequence of $\{\Sigma_k\}_{k\in \mb{N}}$ converges smoothly to a smooth embedded stable FBMH $\Sigma$ in $M$. Note that $\Sigma$ must be one-sided. Moreover, the convergence must be of multiplicity one. Otherwise, we would have that convergence is of multiplicity two and $\Sigma_k$ is eventually diffeomorphic to the two-sided covering of $\Sigma$, which contradicts the assumption that $\{\Sigma_k\}_{k\in \mb{N}}$ is in $\mc{U}_S$. Hence, we have $2\area(\Sigma)=\lim_{k\to \infty} 2\area(\Sigma_k)=\mc{A}_\mathcal S(M,\partial M)$.

This completes the proof.
\end{proof}

As a corollary, we have the following:


\begin{corollary}\label{cor:lower}
Let $M^{n+1}$ be a smooth compact manifold with boundary $\partial M$ and $2\leq n\leq 6$. Then $\mc{A}_{\mathcal S}(M,\partial M)$ is lower-semi continuous with respect to $C^q$ ($q\geq 2$) metrics on $M$.
\end{corollary}
\begin{proof}
Let $g$ be a $C^q$ metric on $M$. Suppose that $\{g_k\}$ is a sequence of $C^q$ metrics on $M$ converging to $g$. If we have
\[\liminf_{k\to \infty}\mc{A}_{\mathcal S}(M,\partial M, g_k)=\infty,\] then the conclusion follows directly. Hence, we may assume that
\begin{equation}\label{equ:AS}
\liminf_{k\to \infty}\mc{A}_{\mathcal S}(M,\partial M, g_k)=C_0
\end{equation}
where $C_0$ is a constant. By (\ref{equ:AS}), there exists a subsequence of $\{g_k\}$ (still denoted by $\{g_k\}$) such that
\begin{equation}\label{equ:AS:C}
\lim_{k\to \infty}\mc{A}_{\mathcal S}(M,\partial M, g_k)=C_0.
\end{equation}
By Theorem \ref{ques:stable:realize}, we can assume that each $\mc{A}_{\mathcal S}(M,\partial M, g_k)$ is realized by a stable FBMH $\Sigma_k$ in $(M,\partial M,g_k)$. Moreover, using (\ref{equ:AS:C}), we can assume that $\sup_{k}\area(\Sigma_k)$ is finite. 
Let $A_k$ denote the second fundamental form of $\Sigma_k$ with respect to $g_k$. Since $\Sigma_k$ is stable and has uniform area bound, by the curvature estimates in \cite{GLZ16}, we obtain that $|A_k|$ is uniformly bounded. Then, up to a subsequence, $\Sigma_k$ converges smoothly to a stable FBMH $\Sigma$ in $(M,\partial M,g)$.

If $\Sigma$ is two-sided, then we have
\[\mc{A}_{\mathcal S}(M,\partial M, g)\leq \area(\Sigma)=\lim_{k\to \infty} \area(\Sigma_k)=\lim_{k\to \infty}\mc{A}_{\mathcal S}(M,\partial M, g_k)=C_0,\]
which implies that
\[\liminf_{k\to \infty}\mc{A}_{\mathcal S}(M,\partial M, g_k)\geq\mc{A}_{\mathcal S}(M,\partial M, g).\]
If $\Sigma$ is one-sided, then for $k$ sufficiently large, either $\Sigma_k$ is one-sided or $\Sigma_k$ is an entire two-sheeted graph over $\Sigma$. In both cases, we have
\[\mc{A}_{\mathcal S}(M,\partial M, g)\leq 2\area(\Sigma)=\lim_{k\to \infty} 2\area(\Sigma_k)=\lim_{k\to \infty}\mc{A}_{\mathcal S}(M,\partial M, g_k)=C_0,\]
which also implies that $\mc{A}_{\mathcal S}(M,\partial M)$ is lower semi-continuous.
\end{proof}

\section{Min-max theory for compact manifolds with boundary}\label{sec:min-max}
\subsection{Geometric Measure Theory}
First let us recall some basic notations from geometric measure theory; see \citelist{\cite{Sim83} \cite{LZ16}*{\S 3} \cite{LMN16}*{\S 2}}.

Let $(M^{n+1},\partial M,T)$ be a Riemannian manifold with boundary and portion. Assume that $(M,\partial M,T)$ is isometrically embedded in some $\R^N$ for $N$ large enough. Let $\mathcal R_k(M)$ be the space of integer rectifiable $k$-currents supporting in $M$. Set
\begin{gather*}
Z_k(M,\partial M)=\{K\in\mathcal R_k(M): \mathrm{spt}(\partial K)\subset\partial M\}.
\end{gather*}
$K$ and $S$ are \emph{in the same equivalent class} if $K,S\in Z_k(M,\partial M)$ and $K-S\in \mathcal R_k(\partial M)$. Denote $\mathcal Z_k(M,\partial M)$ as all the equivalent classes in $Z_k(M,\partial M)$. For any $\kappa\in \mathcal Z_k(M,\partial M)$, there is a \emph{canonical representative} $K\in \kappa$ such that $K\llcorner\partial M=0$.

An integer rectifiable current $K\in \mathcal R_k(M)$ is called an {\em integral current} if and only if $\partial K\in \mathcal R_{k-1}(M)$. The space of integral currents is denoted by $\mathbf I_k(M)$.

For any $K\in\mathcal R_k(M)$, denote  $|K|, \Vert K\Vert$ as the integer rectifiable varifolds and Radon measures associated with $K$, respectively. Given any hypersurface $\Sigma$ with possible non-empty boundary or an open set $\Omega\subseteq M$, we denote $\llbracket \Sigma\rrbracket ,\llbracket \Omega\rrbracket $, and $[\Sigma],[\Omega]$ as the integral currents and integral varifold, respectively.

Denote $\mathbf M$ as the mass norm on $\mathcal R_k(M)$ and $\mathcal F$ as the flat metric on it. In the space of relative cycles, the flat metric and mass norm are defined to be
\begin{gather*}
\mathcal F(\kappa_1,\kappa_2)=\inf\{\mathcal F(K_1,K_2):K_1\in\kappa_1,K_2\in \kappa_2\},\\
\mathbf M(\kappa)=\inf\{\mathbf M(K):K\in\kappa\}.
\end{gather*}

\subsection{Almgren-Pitts Settings}
In this part, we recall Almgren-Pitts min-max theory for compact manifolds with boundary, which is developed by Li-Zhou \cite{LZ16}.

In the following of the paper, we will focus on the one-parameter case, and the notations for cell complex will be restricted to this case.

In this part, $M$ is always a compact manifold with boundary $\partial M$ and portion $T$.
\begin{definition}[\cite{Zhou15}*{Definition $4.1$}]Set $I=[0,1]$.
\begin{enumerate}
  \item The $0$-complex $I_0=\{[0],[1]\}$;
  \item For any $j\in\mathbb N$, $I(1,j)$ has $0$-cells $\{[\frac{i}{3^j}]\}$ and $1$-cells $\{[\frac{i}{3^j},\frac{i+1}{3^j}]\}$ for $0\leq i\leq 3^j$. We always denote $I(1,j)_p$ as the set of $p$-cells of $I(1,j)$;
  \item Given $\alpha\in I(1,j)_1$, we denote $\alpha(k)_p$ as the $p$-complex of $I(1,j+k)$ contained in $\alpha$;
  \item The boundary homeomorphism $\partial: I(1,j)_1\rightarrow I(1,j)_0$ is $\partial[a,b]=[b]-[a]$;
  \item The distance function $\mathbf d:I(1,j)_0\times I(1,j)_0\rightarrow \mathbb N$ is $\mathbf d(x,y)=3^j|x-y|$.
  \item The map $\n(i,j):I(1,i)_0\rightarrow I(1,j)_0$ is defined to be the the way: for each $x\in I(1,i)_0$, $\n(i,j)(x)$ is the unique element of $I(1,j)_0$ such that
  \[\mathbf d(x,\n(i,j)(x))=\inf\{\mathbf d(x,y):y\in I(1,j)_0\}.\]
\end{enumerate}
\end{definition}


Let $A$ and $B$ be two subsets of $\mathcal Z_n(M,\partial M)$. $\phi$ is said to be a map into $(\mathcal Z_n(M,\partial M),A,B)$ if $\phi(0)\in A$ and $\phi(1)\in B$. $\phi$ is said to be a map into $(\mathcal Z_n(M,\partial M),A)$ if $\phi(0),\phi(1)\in A$.

The following homotopy relations were introduced by Pitts \cite{Pit76}*{\S 4.1}.
\begin{definition}[Homotopy for mappings]
Given two maps
\[\phi_i:I(1,j_i)_0\rightarrow (\mathcal Z_n(M,\partial M),\{\llbracket T\rrbracket\},\{0\})\]
for $i=1,2$ and $\delta>0$,
we say that $\phi_1$ \emph{is} $1$-\emph{homotopic to} $\phi_2$ \emph{with $\mathbf M$-fineness} $\delta$ if there exists $j_3>j_1,j_2$ and
\begin{equation*}
\psi: I(1,j_3)_0\times I(1,j_3)_0\rightarrow \mathcal Z_n(M,\partial M),
\end{equation*}
with
\begin{itemize}
  \item $\mathbf f_{\mathbf M}(\psi)\leq\delta$ with definition in \cite{LZ16}*{Definition 4.2};
  \item $\psi(i-1,x)=\phi_i(\n(j_3,j_i)(x)),i=1,2$;
  \item $\psi(I(1,j_3)_0\times \{0\})=\llbracket T\rrbracket$ and $\psi(I(1,j_3)_0\times\{1\})=0$.
\end{itemize}
\end{definition}

\begin{definition}
For a sequence of
\begin{equation*}
\phi_i:I(1,j_i)_0\rightarrow\big(\mathcal Z_n(M,\partial M),\{\llbracket T\rrbracket\},\{0\}\big),
\end{equation*}
$\{\phi_i\}$ is a $(1,\mathbf M)$-\emph{homotopy sequence of mappings into} $\big(\mathcal Z_n(M,\partial M),\{\llbracket T\rrbracket\},\{0\}\big)$ if $\phi_i$ is $1$-homotopic to $\phi_{i+1}$ with fineness $\delta_i\rightarrow 0$, and
\begin{equation*}
\sup_{i}\{\mathbf M(\phi_i(x)):x\in\mathrm{dmn}\phi_i\}<\infty.
\end{equation*}
\end{definition}

\begin{definition}[Homotopy for sequences of mappings]
Let $S_1=\{\phi_i^1\}$ and $S_2=\{\phi_i^2\}$ be two $(1,\mathbf M)$-homotopy sequences of mappings into $(\mathcal Z_n(M,\partial M),\{\llbracket T\rrbracket\},\{0\})$, we say that $S_1$ is homotopic to $S_2$ if $\phi_i^1$ is $1$-homotopic to $\phi_i^2$ with fineness $\delta_i\rightarrow 0$.
\end{definition}

Now assume that $T=\emptyset$. Denote $\pi^\sharp_1(\mathcal Z_n(M,\partial M,\mathbf M),\{0\})$ as the set which consists of all equivalent classes of $(1,\mathbf M)$-homotopy sequences of mappings into $(\mathcal Z_n(M,\partial M),\{0\})$. Similarly, we can define $\pi^\sharp_1(\mathcal Z_n(M,\partial M,\mathcal F),\{0\})$.

In \cite{Alm62}*{\S 3.2}, Almgren introduced a map from the space of equivalent classes to the top relative homology group:
\[ F:\pi^\sharp_1\left(\mathcal Z_n\left(M,\partial M,\mf{M}\right),\{0\}\right)\rightarrow H_{n+1}\left(M,\partial M\right).\]
In fact, $F$ is defined by adding all $\mathbf M$-isoperimetric choices (see \cite{LZ16}*{\S 3.2}) between adjacent slices of $\phi_i$ (for $i$ large enough). Almgren \cite{Alm62} further proved that $F$ is an isomorphism. We usually call $F$ the {\em Almgren's Isomorphism}.

For $\Pi\in\pi^\sharp_1(\mathcal Z_n(M,\partial M,\mf{M}),\{0\})$, and $S=\{\phi_i\}\in\Pi$, define
\begin{equation*}
\mathbf L(S)=\limsup_{i\rightarrow\infty}\max_{x\in\mathrm{dmn}\phi_i}\mathbf M\left(\phi_i\left(x\right)\right),
\end{equation*}
and the \emph{width of} $\Pi$
\begin{equation*}
\mathbf L(\Pi)=\inf_{S\in\Pi}\mathbf L(S).
\end{equation*}
\begin{theorem}[\cite{LZ16}*{Theorem 4.21, Theorem 5.2}]\label{thm:min-max}
Let $(M^{n+1},\partial M,g)$ be a compact manifold and $2\leq n\leq 6$. For any $\Pi\in\pi_1^\sharp\left(\mathcal Z_n\left(M,\partial M,\mathbf M\right),\{0\}\right)$, there exists a varifold $V$ such that
\begin{enumerate}
  \item $\Vert V\Vert(M)=\mathbf L(\Pi),$
  \item $V=\sum n_i[\Sigma_i]$ as a varifold, where $n_i\in\mathbb N$ and each $(\Sigma_i,\partial\Sigma_i)\subset(M,\partial M)$ is a connected, almost properly embedded free boundary minimal hypersurface.
\end{enumerate}
\end{theorem}

\vspace{0.5em}
Now assume that $T\neq\emptyset$. Using $\mathbf M$-isoperimetric lemma \ref{lemma:M-isoperimetric} and Constancy Theorem \cite{Sim83}*{\S 26.27}, we can still define Almgren's map $F$ from $(1,\mathbf M)$-homotopy sequences into $(\mathcal Z_n(M,\partial M),\{\llbracket T\rrbracket\},\{0\})$ to $\mathbf I_{n+1}(M)$.
Precisely, we have the following:
\begin{lemma}\label{lemma:portion imply sweep}
Suppose that $T\neq \emptyset$. For each $(1,\mathbf M)$-homotopy sequence $S=\{\phi_i\}$ mapping into $\left(\mathcal Z_n\left(M,\partial M\right),\{\llbracket T\rrbracket\},\{0\}\right)$, $F(S)=-\llbracket M\rrbracket$. Moreover, all such $S$ are equivalent.
\end{lemma}
\begin{proof}
By Lemma \ref{lemma:M-isoperimetric} and definition of $F$,
\[\mathrm{spt}\left(\partial F(S)+\llbracket T\rrbracket\right)\subset\partial M.\]
Now using the Constancy Theorem \cite{Sim83}*{\S 26.27}, we conclude that $F(S)=-\llbracket M\rrbracket$.

Now let us prove that all such $S$ are equivalent. Let $S=\{\phi_i\}$ and $S=\{\phi_i'\}$ be two $(1,\mathbf M)$-homotopy sequences. Set
\[\widetilde\phi_i:=\phi_i-\phi_i'.\]
Then $\widetilde S=\{\widetilde\phi_i\}$ is a $(1,\mathbf M)$-homotopy sequence mapping into $\left(\mathcal Z_n\left(M,\partial M\right),\{0\}\right)$. Moreover, $F(\widetilde S)=0$. Now by Almgren's Isomorphism, $\widetilde S$ is homotopic to $\{0\}$ in $\left(\mathcal Z_n\left(M,\partial M\right),\{0\}\right)$. Hence, there exists a sequence $\delta_i\rightarrow 0$ such that $\phi_i-\phi_i'$ is homotopic to 0 with fineness $\delta_i$ in $\left(\mathcal Z_n\left(M,\partial M\right),\{0\}\right)$. It follows that $\phi_i$ is 1-homotopic to $\phi_i'$ with $\mathbf M$-fineness $\delta_i$ in $(\mathcal Z_n(M,\partial M),\{\llbracket T\rrbracket\},\{0\})$. This completes our proofs.
\end{proof}

\subsection{Sweepouts}\label{subsec:sweepout}

\begin{definition}
A map
\[\Phi:I\rightarrow\mathcal Z_n(M,\partial M)\]
is called a \emph{regular one-parameter family of $(M,\partial M,T)$} if it satisfies the following:
\begin{enumerate}
  \item $\Phi$ is continuous in flat topology;
  \item $\sup_{x\in I}\M(\Phi(x))<+\infty$;
  \item there is no mass concentration on $\Phi$: $\lim_{r\rightarrow\infty}\m(\Phi,r)=0$ (see \cite{LZ16}*{Definition 4.11} and \cite{MN14}*{\S 4.2} for the definition of $\m(\Phi, r)$).
\end{enumerate}
\end{definition}

Denote $\mathcal P_1(M)$ as the collection of regular one-parameter families.

Assume that $T=\emptyset$. Given $\Phi\in\mathcal P_1$ with $\Phi|_{\{0,1\}}=0$, there is a $(1,\mathbf M)$-homotopy sequence $S_{\Phi}$ mapping into $(\mathcal Z_n(M,\partial M,\mf{M}),\{0\})$ by Discretization Theorem \ref{thm:discrete}. Denote $\Pi_{\Phi}$ be the equivalent class of $S_{\Phi}$ in $\pi^\sharp_1(\mathcal Z_n(M,\partial M,\mf{M}),\{0\})$.

\begin{definition}\label{def:sweepout}
If $T=\emptyset$, a regular one-parameter family
\[\Phi:I\rightarrow\left(\mathcal Z_n(M,\partial M),\{0\}\right)\]
is called a \emph{sweepout of $(M,\partial M)$} if and only if $F(\Pi_{\Phi})$ represents non-zero element in $H_{n+1}(M,\partial M)$.
\end{definition}
\begin{remark}
 By min-max theorem \ref{thm:min-max}, a sweepout can always produce FBMHs with multiplicities, which is called \emph{min-max minimal hypersurface  corresponding to the fundamental class}.
\end{remark}

\begin{definition}[Sweepout of manifolds with portions]\label{def:sweepout with portion}
If $T\neq \emptyset$. A regular one parameter family
\[\Phi:I\rightarrow(\mathcal Z_n(M,\partial M),\{\llbracket T\rrbracket\},\{0\})\]
is called a \emph{sweepout of $(M,\partial M,T)$}.
\end{definition}

\begin{proposition}\label{prop:sweepout:B to NB}
A sweepout of $(M,\partial M,T)$ is also a sweepout of $(M,\partial M\cup T)$.
\end{proposition}
\begin{proof}
Let $\Phi$ be a sweepout of $(M,\partial M,T)$. Let $\{\phi_i\}$ be the $(1,\mathbf M)$-homotopy sequence mapping into $\left(\mathcal Z_n(M,\partial M\cup T),\{0\}\right)$, which is produced by discretizing $\Phi$. Namely, we can take $\{\phi_i\}$ to be a $(1,\mathbf M)$-homotopy sequence mapping into $\left(\mathcal Z_n(M,\partial M),\{\llbracket T\rrbracket\},\{0\}\right)$ by Discretization Theorem \ref{thm:discrete}.

Applying Lemma \ref{lemma:portion imply sweep}, we have $F(\{\phi_i\})=-\llbracket M\rrbracket$. By Definition \ref{def:sweepout}, $\Phi$ is a sweepout of $(M,\partial M\cup T)$.
\end{proof}

\begin{remark}
We claim that there does exist such sweepouts. Let $r$ be the distance function to $T$, which is defined in a small neighborhood of $T$. Perturb it slightly and extend it to whole $M$ such that the extended function $f$ satisfies:
\begin{itemize}
  \item $f$ is a smooth Morse function;
  \item $f^{-1}(0)=T$.
\end{itemize}
It is easy to verify that $\{f^{-1}(t)\}$ is a sweepout.
\end{remark}
Let $\Phi$ be a sweepout of $(M,\partial M,T)$. Then we define 
\[\mathbf L(\Phi):=\sup_{x\in I}\M(\Phi(x)).\]

And \emph{the width of} $(M,\partial M,T)$ is defined as
\[W(M,\partial M,T):=\inf\{\mf{L}\mathbf(\Phi):\Phi \text{ is a sweepout of $(M,\partial M,T)$}\}.\]

\subsection{Min-max Theorem with Barriers}
In this part, we prove that the free boundary min-max theory (Theorem \ref{thm:min-max}) can also produce a FBMH when there is a barrier. This can be viewed as a free boundary version of \citelist{\cite{MN12}*{Theorem 2.1}\cite{Zhou15}*{Theorem 2.7}\cite{Son15}*{Theorem 13}\cite{Wang}*{Theorem 3.4}} using discrete sweepouts.

Let $M$ be a compact manifold with boundary $\partial M$ and portion $\Sigma$. We further assume that $\Sigma$ is a barrier in the sense of Definition \ref{D:manifold with boundary and portion}. Let $r$ be the distance function to $\Sigma$. Apparently, there is a closed manifold $\widetilde M$ and a closed hypersurface $\widetilde\Sigma$ such that $(M,\Sigma)$ can be embedded into $\widetilde M$ satisfying $M\cap\widetilde\Sigma=\Sigma$. Let $\widetilde r$ be the distance function to $\widetilde \Sigma$. Hence, $r=\widetilde r\big|_M$ in a small neighborhood of $\Sigma$ in $M$. Since $\Sigma$ is a barrier, $r$ is well-defined and smooth in a small neighborhood of $\Sigma$. Set
\[M_s=\{x\in M:r(x)>s\}.\]

Now since $\Sigma$ is a barrier, there is a constant $a>0$ such that $r^{-1}(t)$ are barriers for all $t<2a$. By taking $a$ small enough, $r^{-1}(t)$ can be assumed to be isotopic to $\Sigma$ for $t<2a$.

\begin{lemma}\label{lem:push forward}
For any sweepout $\Phi$ of $(M,\partial M,\Sigma)$ and $t_0\in(0,1)$, there exists another sweepout $\Phi'$ satisfying
\begin{enumerate}
  \item $\Phi(0)=\Phi'(0)$;
  \item \label{area non-increasing} $\M(\Phi'(t))\leq\M(\Phi(t))$ for all $t\in[0,1]$;
  \item \label{away} $\mathrm{spt}(\Phi'(t))\subseteq M_\frac{a}{2}$ for all $t>t_0$.
\end{enumerate}
\end{lemma}
\begin{proof}
Let $A$ be the second fundamental form of the level set of $r$. Set
\[ c=\sup\{|A(x)|:r(x)\leq 2a\}<+\infty.\]

Let $\phi$ be some cut-off function satisfying
\begin{itemize}
  \item $\phi'+c\phi\leq 0$;
  \item $\phi(r)=0$, for all $r>2a$;
\end{itemize}
The existence of such $\phi$ is shown in \cite{MN12}*{Lemma 2.2}.

Denote $(G_t)_{0\leq t\leq1}$ as the one-parameter family of homomorphisms of $\widetilde M$ generated by $\phi\nabla\widetilde r$. Let $L$ be the canonical representative of  $\tau\in\Z_n(M,\partial M)$. By directly computation (see \cite{Wang} for details),
\[ \mathrm{div}_{(G_t)_\sharp L}(\phi\nabla r)\leq 0.\]

Notice that $(G_s)_\sharp(L)$ is an integer rectifiable $n$-current in $\widetilde M$. By the first variation formula,
\[\frac{d}{ds}\M\left(\left(G_s\right)_\sharp(L)\right)=\int_{G_{s\sharp} L}\mathrm{div}_{\left(G_s\right)_\sharp(L)}\left(\phi\nabla r\right) \leq 0.\]

This implies that for $t\in(0,a)$
\begin{equation}\label{area decrea}
\M\left(\left(G_{t\sharp} L\right)\llcorner M\right)\leq \M\left(G_{t\sharp} L\right)\leq\M\left(L\right),
\end{equation}
for each $L\in Z_n(M,\partial M)$. Notice that $r^{-1}(t)$ are all barriers for $t<2a$ and $\phi$ is supported on $[0,2a]$. Recalling the choice of $\phi$, $G_s(p)=p$ for $p\in\partial M\cap M_a$, and $(G_{s\sharp})L$ is an element in $Z_n(\widetilde M,\partial M)$.

Now for any $\kappa\in \mathcal Z_n(M,\partial M)$ with canonical representative $K$, define $G_{s\sharp}\kappa$ to be the equivalent class of $G_{s\sharp}K$.

Let $S>0$ be such that $G_S(\Sigma)\cap M=r^{-1}(\frac{a}{2})$ and then choose a smooth non-negative function $h:[0,1]\rightarrow [0,S]$ such that $h(0)=0$, $h(t)>0$ for $t>0$ and $h(t)=S$ for $t\geq t_0$. Set
\begin{equation*}
\Phi'(t)=\big((G_{h(t)})_\sharp \Phi(t)\big)\llcorner M.
\end{equation*}
Then if $t=0$, $\Phi'(t)=\Phi(t)$; if $t>0$, it follows from (\ref{area decrea}) that
\[ \M((G_{h(t)})_\sharp(\Phi(t)))\leq \M(\Phi(t)).\]

For the last requirement, noticing that $h(t)=S$ for $t\geq t_0$ and the $G_S(\Sigma)\cap M=r^{-1}(\frac{a}{2})$, we conclude that $\mathrm{spt}\Phi'(t)\subseteq M_{\frac{a}{2}}$.

To complete the proof, it suffices to show that $\Phi'$ is a sweepout of $(M,\partial M,\Sigma)$. This follows from Definition \ref{def:sweepout with portion}.
\end{proof}

\begin{theorem}\label{new minimal surface}
Let $M$ be a Riemannian manifold with boundary $\partial M$ and portion $\Sigma$. If $W(M,\partial M,\Sigma)>\area(\Sigma)$, there exists a min-max sequence $\{\Phi_{i_k}^k(t_k)\}$ converging in the varifold sense to an integral varifold which is supported on an embedded free  boundary (possibly empty) minimal hypersurface $\Gamma$ (possibly disconnected), which satisfies $\Gamma\cap\Sigma=\emptyset$. Furthermore, the width satisfies
\begin{equation*}
W(M,\partial M,\Sigma)=\area(\Gamma),
\end{equation*}
if counted with multiplicities.
\end{theorem}
\begin{proof}
It suffices to show that there exists a minimizing sequence $\{\Phi^k(t)\}_{k=1}^\infty$ of sweepouts of $(M,\partial M,\Sigma)$ such that
\begin{equation}\label{pull}
\M(\Phi^k(t))\geq W(M,\partial M,\Sigma)-\delta\Rightarrow\mathrm{dist}\left(\mathrm{spt}\left(\Phi^k(t)\right),\Sigma\right)\geq\frac{a}{2},
\end{equation}
where $\delta=\frac{1}{4}\left(W(M,\partial M,\Sigma,\Lambda)-\area(\Sigma)\right)>0$, and $\mathrm{dist}(\cdot,\cdot)$ is the distance function of $(M,\partial M,g)$.

For any sweepout $\{\Psi^k(t)\}$, there always exists $\epsilon_k>0$ such that
\begin{equation}\label{local depart}
\M(\Psi^k(t))\leq\area(\Sigma)+\delta\ \ \text{ for all } t\in[0,2\epsilon_k].
\end{equation}
Then by taking $t_0=\epsilon_k$ in Lemma \ref{lem:push forward}, there is a better sweepout $\{\widetilde\Psi^k(t)\}$, which satisfies (\ref{pull}). In fact,
\begin{equation*}
\M(\widetilde\Psi^k(t))\geq W(M,\partial M,\Sigma)-\delta
\end{equation*}
implies that
\begin{equation*}
\M(\widetilde\Psi^k(t))\geq W(M,\partial M,\Sigma)-\delta\geq \area(\Sigma)+\delta.
\end{equation*}
Then by (\ref{local depart}) and Lemma \ref{lem:push forward} (\ref{area non-increasing}), we have $t\geq 2\epsilon_k$. Now using Lemma \ref{lem:push forward} (\ref{away}), we conclude that
\begin{equation*}
\mathrm{dist}(\mathrm{spt}\widetilde\Psi^k(t),\Sigma)\geq\frac{a}{2}.
\end{equation*}

Now modifying the arguments of min-max theory for compact manifolds with boundary in \cite{LZ16}, we can obtain a FBMH $(\Gamma,\partial\Gamma)$ with $\partial\Gamma\subseteq\partial M$. Let us sketch the main steps here.

Let $\{\{\widetilde\Psi^k(t)\}_{t\in[0,1]}\}_{k=1}^\infty$ be the minimizing sequence satisfying (\ref{pull}). By Discretization Theorem \ref{thm:discrete}, for each $\widetilde\Psi^k$, there exists a sequence of mappings
\[\psi^k_i:I(1,j^k_i)_0\rightarrow\left(\mathcal Z_n(M,\partial M),\{\llbracket\Sigma\rrbracket\},\{0\}\right)\]
with $j^k_i<j^k_{i+1}$ and a sequence of positive numbers $\delta^k_i$ satisfying requirements in Theorem \ref{thm:discrete}. Applying Interpolation Theorem \ref{thm:interpolation}, there exists a sequence of continuous maps:
\[\overline\Psi_i^k:I\rightarrow \mathcal Z_n(M,\partial M;\mathbf M)\]
such that $\overline\Psi_i^k(x)=\psi_i^k(x)$ for all $x\in I(1,j^k_i)_0$. Recalling that $\psi_i^k(0)=\llbracket\Sigma\rrbracket$, we conclude that each $\overline\Psi_i^k$ is a sweepout of $(M,\partial M,\Sigma)$.

Next we follow the tightening process in \cite{LZ16}*{Proposition 4.17}, where each sequence of mappings can be deformed to another  $\{\Phi^k_i(t)\}$ such that each min-max sequence converges to a stationary varifold. Since those $\psi_i^k(t)$ with mass closing to $W(M,\partial M,\Sigma)$ have a distance $a/2>0$ away from $\Sigma$, we can fix the mappings near $\Sigma$ in the tightening process. Therefore $\Phi^k_i$ is also a sweepout of $(M,\partial M,\Sigma)$.

Now for an almost minimizing min-max sequence $\{\Phi^k_{i_k}(t_k)\}$, it follows that $\mathrm{spt} \left(\Phi^k_{i_k}(t_k)\right)$ always have a distance $a/2$ away from $\Sigma$ for large $k$ by (\ref{pull}).

Finally, we show that the limit of the almost minimizing min-max sequence is supported on some embedded FBMHs. These were done by Li-Zhou \cite{LZ16}*{\S 10.3, \S 5} and there are no differences here. Hence, there is a FBMH $(\Gamma,\partial\Gamma)$  with $\partial\Gamma\subseteq\partial M$. Since the supports of the minimizing sequence have fixed distance away from $\Sigma$, we conclude that $\Gamma\cap\Sigma=\emptyset$.
\end{proof}

\section{Free boundary Minimal hypersurfaces with area less than \texorpdfstring{$\mc{A}_{\mathcal S}(M,\partial M)$}{ASM}}\label{sec:unstable}
In this section, we study FBMHs whose areas are less than $\mc{A}_{\mathcal S}(M,\partial M)$.

Let $(\Sigma,\partial\Sigma)$ be an almost properly embedded hypersurface in $(M,\partial M)$. Take a cut-off function $\phi$ which is supported in a neighborhood of the touching set of $\Sigma$ such that $\langle \phi\n,\nu_{\partial M}\rangle <0$ on touching set, where $\n$ is the normal vector field of $\Sigma$. Set 
\[\Sigma_{t\phi}:=\{\mathrm{exp}_x(t\phi\n):x\in\Sigma\}.\]
This is well-defined for $t>0$ small enough by the definition of $\phi$. Then $\Sigma_{t\phi}$ is called a \emph{generic type of} $\Sigma$ (see Figure V). Obviously, all the generic types of $\Sigma$ are isotopic to each other. Moreover, they are properly embedded hypersurfaces.

$\Sigma$ is said to \emph{generically separate} $M$ if there is a generic type of $\Sigma$ separating $M$.

\begin{remark}
We emphasis that the generic type is also well-defined for non-orientable hypersurfaces. In fact, all we need is a well-defined locally normal vector field. Notice that  a neighborhood of touching set can be seen as a graph over $\partial M$; hence there is a neighborhood of the touching set in $\Sigma$ which is orientable.
\end{remark}

\begin{figure}[h]
\begin{center}
\def\svgwidth{0.7\columnwidth}
  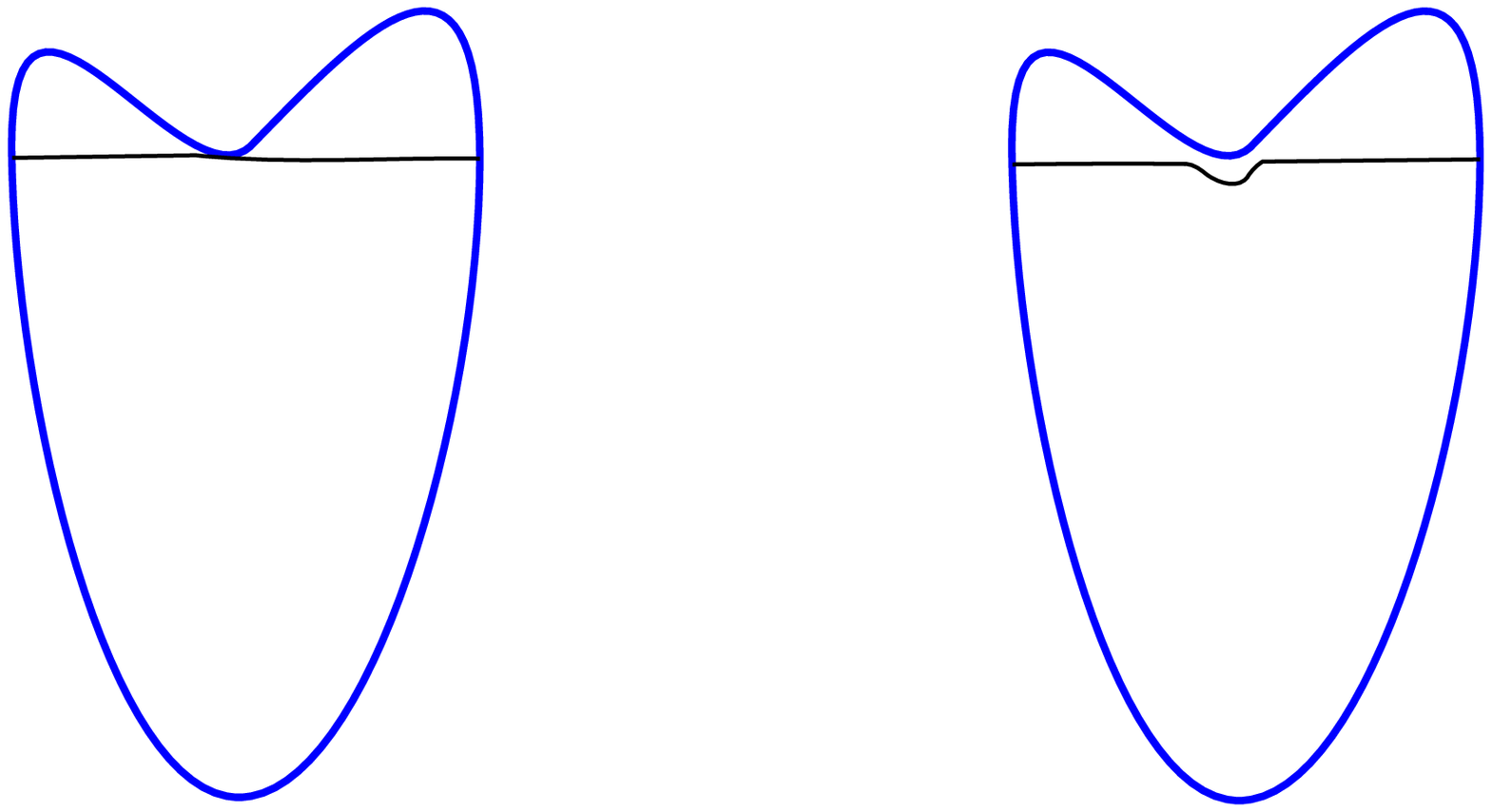
  \label{graph:generic}
  \caption{}
\end{center}
\end{figure}

\subsection{Construction of Stable Minimal Hypersurfaces}
\begin{proposition}[cf. \cite{MaRo17}*{Proposition 14}]\label{prop:non-orient:less}
Suppose that $\Sigma$ is a non-orientable almost properly embedded free boundary minimal hypersurface in $M$. Then there is a connected, properly embedded, stable free boundary minimal hypersurface $\Sigma'$ in $M$ such that $\area(\Sigma')\leq \area(\Sigma)$.
\end{proposition}
\begin{proof}
Since $\Sigma$ is non-orientable, each generic type of $\Sigma$ is also non-orientable and represents a non-zero element in $H_n(M,\partial M;\mb{Z}_2)$. We can then minimize the mass norm among all the relative chains that are homologous to $[\Sigma]$. Using the regularity theory (see \cite{Mor86} for the interior regularity and \cite{Gru87} for boundary regularity), we obtain a smooth properly embedded FBMH $\Sigma'$ in $M$ which minimizes the area in the homology class represented by generic types of $\Sigma$. Since $\Sigma'$ is a minimizer, it must be stable and $\area(\Sigma')$ is less than or equal to the area of each generic type of $\Sigma$. As a generic type of $\Sigma$ can be produced by perturbing $\Sigma$ very slightly, it follows that $\area(\Sigma')\leq \area(\Sigma)$.
\end{proof}


\begin{proposition}\label{minimizer}
Let $(M, \partial M, T)$ be a compact manifold with smooth boundary $\partial M$ and portion $T$ such that $T$ is a barrier. Let $\Sigma$ be a properly embedded orientable hypersurface in $M$ with $\partial\Sigma\subset \partial M$. If $[\Sigma]$ represents a nontrivial element in $H_n(M,\partial M;\mathbb Z)$, then there exists a properly embedded, stable, orientable, free boundary minimal hypersurface $S$ with boundary in $\partial M$ such that $\area(S)\leq\area(\Sigma)$. Moreover, the equality holds only if $\Sigma$ is a stable free boundary minimal hypersurface.
\end{proposition}
\begin{proof}
In terms of geometric measure theory, $\Sigma$ can be seen as a relative integral $n$-cycle. Suppose that $\{\tau_i\}$ is a sequence of relative cycles such that
\begin{itemize}
  \item each $\tau_i$ is homologous to $[\Sigma]$ in $H_n(M,\partial M)$;
  \item $\lim_{i\rightarrow\infty}\M(\tau_i)=\inf \{\M(\kappa):\kappa\in \mathcal Z_n(M,\partial M),\ \kappa \text{ is homologous to } [\Sigma]\}$
\end{itemize}
Since $T$ is a barrier, each $\tau_i$ can be assumed to be away from $T$ with fixed distance. By the compactness theorem for relative cycles \cite{LZ16}*{Lemma 3.10} (see also \cite{LMN16}*{\S 2.3}), $\tau_i$ converges to $\tau\in\mathcal Z_n(M,\partial M)$ in flat topology. Moreover, $\tau_i$ is homologous to $[\Sigma]$, which can be seen directly from \cite{LZ16}*{Lemma 3.15} (see also \cite{Mor16}*{\S 12.3}). Since it is a minimizer, by the regularity theory \cite{Gru87}, $\tau$ is supported on a properly embedded, stable FBMH which may not be connected. Then the proposition follows by taking any component.
\end{proof}

\subsection{Construction of Sweepouts}
\begin{proposition}\label{prop:sweepout}\label{prop:less_area:sweepout}
Let $(\Sigma,\partial\Sigma)\subset(M,\partial M)$ be an orientable, almost properly embedded, free boundary minimal hypersurface with $\area(\Sigma)<\mc{A}_{\mathcal S}(M,\partial M)$. Then there is a sweepout
\[\Psi:I\longrightarrow \mathcal Z_n(M,\partial M),\]
such that
\begin{enumerate}
  \item $\Psi(\frac{1}{2})=\Sigma;$
  \item $F(\Psi)=\llbracket M\rrbracket$ (see  \S \ref{subsec:sweepout});
  \item $\M(\Psi(t))<\area(\Sigma)$ for $t\neq 0$.
\end{enumerate}
\end{proposition}

\begin{proof}[Proof of Proposition \ref{prop:less_area:sweepout}]
Let $X\in\mathfrak X(M,\Sigma)$ be a vector field such that $X\big|_{\Sigma}=-f\mf{n}$, where $f$ is the function given in the proof of Proposition \ref{prop:barrier:unstable}. Let $(F_s)_{-\tau\leq s\leq\tau}$ be the family of diffeomorphisms corresponding to $X$.

We first claim that $\Sigma$ generically separates $M$. If not, a generic type of $\Sigma$ represents a non-zero homology class of $H_n(M,\partial M)$. By virtue of Proposition \ref{minimizer} ($T=\emptyset$), there is a stable FBMH which has area less than $\Sigma$. This is not the case of the assumption.

Hence, $M$ can be divided into $M^+$ and $M^-$ (see Figure VI) in the following way: for any $x\in M$, the available paths are differentiable curves $\gamma:[0,1]\rightarrow M$ from $x$ to $\Sigma$ such that $\gamma(0,1)$ is disjoint from $\Sigma$. Then $M^+$ ($M^-$) is defined to be the collection of $x\in M$ such that there are available paths with $r(\gamma(t))>0$ ($<0$) for some $t$.

\begin{figure}[h]
\begin{center}
\def\svgwidth{0.65\columnwidth}
  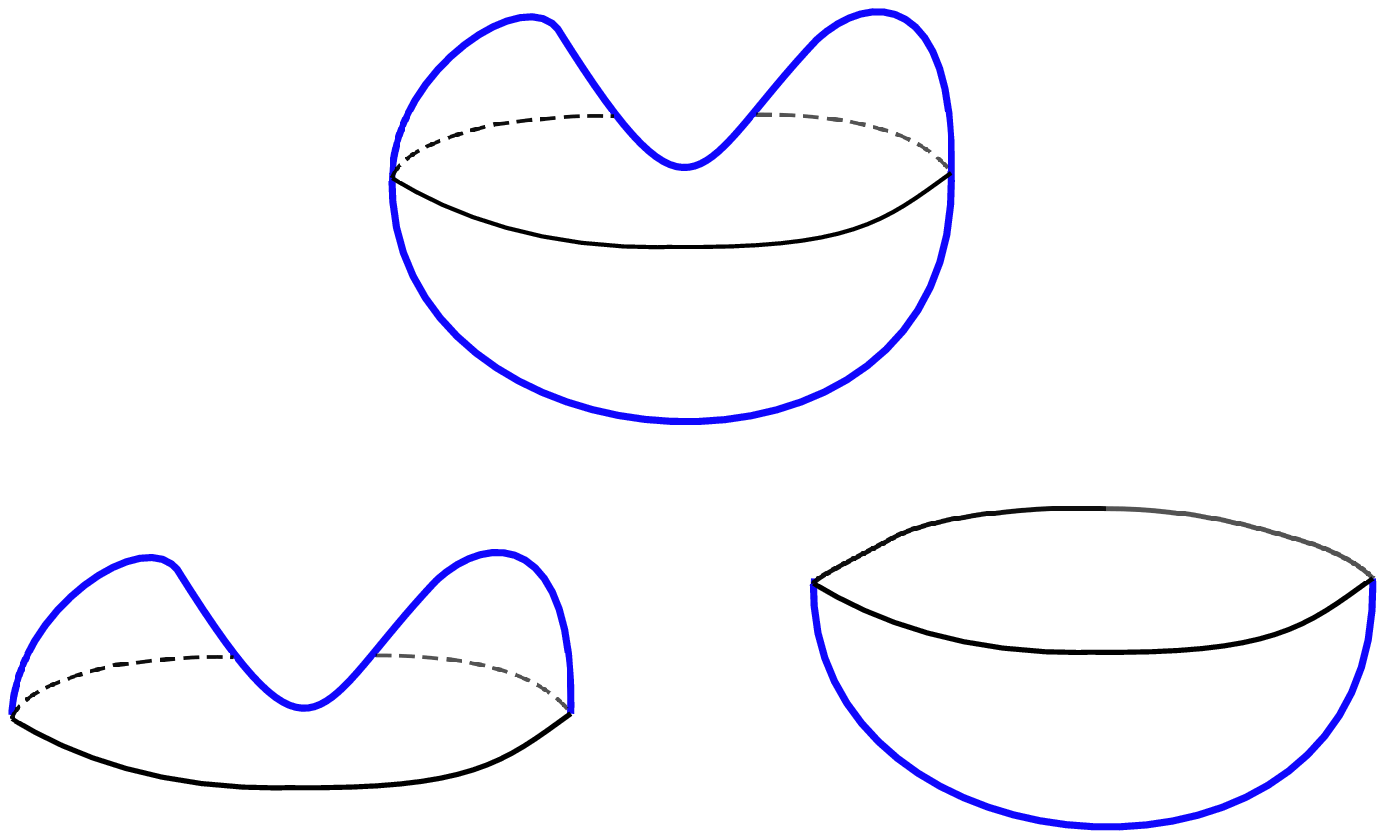
  \label{graph:separate}
  \caption{}
\end{center}
\end{figure}

Define the sweepout locally by
\[\Sigma_s=F_s(\Sigma)\cap M.\]

By the second variation formula of area, it follows that
\[\area(\Sigma_s)<\area(\Sigma),\ \text{for $s\in(0,\tau)$ .}\]

Set
\begin{gather*}
  M^+_s:=M^+\setminus(\cup_{0\leq t<s}\Sigma_t),\\
  \partial M^+_s:=\partial M\cap M^+_s.
\end{gather*}

By  Proposition \ref{prop:barrier:unstable}, we obtain that $\Sigma_\eta$ is a barrier for $\eta$ small enough since $\Sigma$ is unstable.

Now we claim that $(M^+_\eta,\partial M^+_\eta,\Sigma_\eta)$ admits a sweepout
\[\Phi^+:I\rightarrow \left(\mathcal Z_n(M^+_\eta,\partial M^+_\eta),\{\llbracket \Sigma_\eta\rrbracket\},\{0\}\right)\]
satisfying
\[\M(\Phi^+(t))\leq \area(\Sigma_\eta) \text{\ for\ } t\in I.\]
Assume on the contrary that there are no such sweepouts. It follows that
\[W(M^+_\eta,\partial M^+_\eta,\Sigma_\eta)>\area(\Sigma_\eta).\]
By Theorem \ref{thm:bumpy}, the metric on $M$ can be perturbed slightly (still denoted by $M$ with new metric) such that
\begin{itemize}
    \item $\Sigma_\eta$ has positive mean curvature;
    \item $W(M^+_\eta,\partial M^+_\eta,\Sigma_\eta)>\area(\Sigma_\eta)$;
    \item no finite covers of any FBMHs with boundary on $\partial M\cap M^+_\eta$ admit a non-trivial Jacobi field.
\end{itemize}
Moreover, we can assume that $\area(\Sigma)<\mathcal A_{\mathcal S}(M,\partial M)$ holds in the perturbed metric since $\mathcal A_{\mathcal S}(M,\partial M)$ is lower semi-continuous by Corollary \ref{cor:lower}.

By virtue of Theorem \ref{new minimal surface}, there is a FBMH $\Gamma$ in $M^+_\eta$ such that $\partial\Gamma\cap\Sigma_\eta=\emptyset$. Let $\Gamma_i$ be a connected component of $\Gamma$. We prove that $\Gamma_i$ generically separates $M^+_\eta$ by arguing in two cases.

In the first case, $\Gamma_i$ is assumed to satisfy
\begin{equation}
  \area(\Gamma_i)>\area (\Sigma).
\end{equation}
Let us prove it by two steps.
\begin{claim}\label{claim_orientable}
$\Gamma_i$ is orientable.
\end{claim}
We prove this claim by a contradiction argument. We first consider the `manifold' $\widetilde M^+_\eta$ which is constructed by opening $M^+_\eta$ along $\Gamma_i$. Then $\widetilde M^+_\eta$ has boundary $\partial\widetilde M^+$ and two portions: one is $\Sigma_\eta$ which has positive mean curvature and the other $\widetilde \Gamma_i$ is the 2-sheeted covering of $\Gamma_i$ which is a connected minimal hypersurface with free boundary. Since $\widetilde\Gamma_i$ is non-degenerate, it follows that $\lambda_1(\widetilde\Gamma_i)\neq 0$. By Proposition \ref{prop:barrier:unstable} and Proposition \ref{prop:barrier:stable}, there is a family of barriers around $\widetilde\Gamma_i$. It is clear that $\Sigma_\eta$ represents a non-trivial element in $H_n(\widetilde M^+_\eta,\partial\widetilde M^+_\eta;\mathbb Z)$. Applying Proposition \ref{minimizer}, there is an orientable, stable FBMH $S$ in $\widetilde M^+_\eta$ with $\area(S)<\area(\Sigma_\eta)$ since it is an area minimizer of the homology class $[\Sigma_\eta]$. It follows that $\area(S)<\area(\Sigma)$. This implies that $S\neq\widetilde\Gamma_i$, and hence $S\cap\Gamma_i=\emptyset$. By the definition of $\widetilde M^+_\eta$, $S$ can also be regarded as a stable minimal hypersurface of $M^+_\eta$. However, this leads to a contradiction with $\area(\Sigma)<\mathcal A_\mathcal S(M,\partial M)$. Claim \ref{claim_orientable} is proved.

\vspace{0.5em}
Similarly, we have
\begin{claim}\label{claim_separable}
$\Gamma_i$ generically separates $M^+_\eta$.
\end{claim}
Assuming on the contrary that $\Gamma_i$ does not generically separate $M^+_\eta$, there is a `manifold' $\overline M^+_\eta$ constructed by opening $M^+_\eta$  along $\Gamma_i$. Then  $\overline M^+_\eta$ has three portions: one is $\Sigma_\eta$ which has positive mean curvature and the other two $\Gamma_p$ and $\Gamma_q$ are both diffeomorphic to $\Gamma_i$. Since $\Gamma_p$ and $\Gamma_q$ have no non-trivial Jacobi fields, there exist a family of barriers around $\Gamma_p$ and $\Gamma_q$. Applying Proposition \ref{minimizer}, there is an orientable, stable minimal hypersurface $S$ in $\overline M^+_\eta$ satisfying
\[\area(S)<\area(\Sigma)<\area(\Gamma_i).\]
Hence, $S$ can not be any component of the portions, which implies that $S$ can be regarded as a stable minimal hypersurface in $M^+_\eta$. This contradicts  $\area(\Sigma)<\mathcal A_\mathcal S(M,\partial M)$. Hence, Claim \ref{claim_separable} is true.

\vspace{0.5em}
In the other case, we have $\area(\Gamma_i)\leq\area(\Sigma)$. We can also prove that $\Gamma_i$ generically separates $M^+_\eta$. Let us argue by contradiction again.
\begin{claim}
  $\Gamma_i$ is orientable.
\end{claim}
If $2\area(\Gamma_i)>\area(\Sigma)>\area(\Gamma_i)$, this can been shown by the same arguments with Claim \ref{claim_orientable}; If $\area(\Sigma)\geq 2\area(\Gamma_i)$, it follows from Proposition \ref{prop:non-orient:less} that there is a stable minimal hypersurface $\Gamma'$ with $\area(\Gamma_i')\leq \area(\Gamma_i)$, which is not the case of assumption.

\vspace{0.25em}
Hence, $\Gamma_i$ is orientable when $\area(\Gamma_i)<\area(\Sigma)$. Assuming on the contrary that it does not generically separate $M^+_\eta$, $\Gamma_i$ represents a non-trivial element in $H_n(M^+,\partial M^+;\mathbb Z)$. By virtue of Proposition \ref{minimizer}, there is a stable, orientable FBMH $\Gamma''$ in $M^+_\eta$ which satisfies $\area(\Gamma'')\leq \area(\Gamma_i)$. This implies
\[\area(\Sigma)\geq\area(\Gamma'')\geq\mathcal A_{\mathcal S}(M,\partial M),\]
which leads to a contradiction.

\vspace{0.7em}
Let us continue our arguments. Overall, $\Gamma_i$ generically separates $M^+_\eta$. Applying Proposition \ref{minimizer} to the component of $M^+_\eta\setminus\Gamma_i$ containing $\Sigma_\eta$, there is an orientable stable minimal hypersurface $S'$ in it with 
\[\area(S')\leq\area(\Sigma_\eta)<\area(\Sigma),\]
which is a contradiction to $\area(\Sigma)<\mathcal A_\mathcal S(M,\partial M)$.

\vspace{0.5em}
Thus we have proved the existence of a good sweepout of $(M^+_\eta,\partial M^+_\eta,\Sigma_\eta)$. By the proof of Proposition \ref{prop:sweepout:B to NB}, it follows that $F(\Phi^+)=-\llbracket M^+_\eta\rrbracket$. Then in $M^-$, we can similarly produce a good sweepout $\Phi^-(t)$ of $(M^-_\eta,\partial M^-_\eta,\Sigma_{-\eta})$. Now define the sweepout as follows
\begin{equation*}
\Phi(t)=\left\{\begin{aligned}
         \Phi^-(1-4t) \ \ \ \ & t\in[0,1/4)\\
         \llbracket\Sigma_{(2t-3/2)\eta}\rrbracket \ \ \ \ & t\in[1/4,1/2)\\
         -\llbracket\Sigma_{(2t-1)\eta}\rrbracket\ \  \ \ &t\in[1/2,3/4)\\
         -\Phi^+(4t-3)\ \ \ \  & t\in[3/4,1].
        \end{aligned}\right.
\end{equation*}

$\Phi$ is continuous in the flat topology of $\mathcal Z_n(M,\partial M)$. Then by the same arguments with Proposition \ref{prop:sweepout:B to NB}, we conclude that $F(\Phi)=\llbracket M\rrbracket$.
\end{proof}

Similar to Proposition \ref{prop:less_area:sweepout}, good sweepouts can also be produced from non-orientable minimal hypersurfaces:
\begin{proposition}\label{prop:sweepout_nonori}
For $\Sigma\in\mathcal U_\mathcal S$ with $2\area(\Sigma)=\mathcal A_{\mathcal S}(M,\partial M)$ such that the 2-sheeted covering of $\Sigma$ is unstable, there is a sweepout
\[\Phi:I\rightarrow\left( \mathcal Z_n(M,\partial M),\{0\}\right),\]
such that
\begin{enumerate}
  \item\label{sweep:fundamental} $F(\Phi)=\llbracket M\rrbracket$;
  \item\label{sweep:maximum} $\M(\Phi(t))<2\area(\Sigma)$ for all $t$;
  \item\label{sweep:foliation} there exists $\epsilon>0$, $\Phi(t)=\llbracket\Sigma_t\rrbracket$ for $t\in (0,\epsilon]$, where $\Sigma_t$ is defined in Proposition \ref{prop:barrier:unstable} in $M\setminus\Sigma$;
  \item\label{sweep:continuous} $|\Phi(t)|$ converges to $2[\Sigma]$ as $t\rightarrow 0$ in the sense of varifolds.
\end{enumerate}
\end{proposition}
\begin{proof}
Denote $\widetilde M$ as the manifold produced by opening up $M$ along $\Sigma$. $\widetilde M$ is a manifold with boundary and a portion $\widetilde\Sigma$, which is a double cover of $\Sigma$. By above arguments, $\widetilde M$ admits a sweepout
\[\widetilde\Phi:I\longrightarrow\mathcal Z_n(\widetilde M,\partial\widetilde M,\widetilde \Sigma).\]
Therefore, $M$ admits a sweepout
\[\Phi:I\longrightarrow \mathcal Z_n(M,\partial M),\]
which can be defined by
\begin{itemize}
  \item $\Phi(0)=0$;
  \item $\Phi(t)=\widetilde \Phi(t)$.
\end{itemize}
It is easy to verify that $\Phi$ is the sweepout in the proposition.
\end{proof}

By the Catenoid Estimates \cite{KMN16}, such a sweepout can be deformed to another one which has less maximal slice:
\begin{proposition}\label{prop:non-min-max}
Suppose that $\Sigma\in \mathcal U$ and there is a sweepout $\Phi$ of $(M,\partial M)$ satisfying (\ref{sweep:fundamental})(\ref{sweep:maximum})(\ref{sweep:foliation})(\ref{sweep:continuous}) in Proposition \ref{prop:sweepout_nonori}. Then
 \[\inf_{F(\Pi)=\llbracket M\rrbracket}\mathbf L(\Pi)<2\area (\Sigma).\]
As a corollary, $2\Sigma$ can not be the min-max minimal hypersurface corresponding to the fundamental class.
\end{proposition}
\begin{proof}
The key point here is to construct another sweepout $\Phi'$ of $(M,\partial M)$ satisfying
\begin{enumerate}
  \item $F(\Phi)=\llbracket M\rrbracket$;
  \item $\Phi'(t)=\Phi(t)$ for all $t>\epsilon$;
  \item $\sup_{t\in I}\M(\Phi'(t))<2\area(\Sigma)$.
\end{enumerate}
The construction here is similar to \cite{KMN16}. We also refer to \cite{Wang}*{Appendix C} for the case in manifolds with boundary.
\end{proof}

\section{Proof of the Main Theorem}\label{sec:proof}
In this section, we give the proof of Main Theorem:
\begin{theorem}
Let $(M^{n+1},\partial M,g)$ be a compact manifold with boundary and $2\leq n\leq 6$. Then there exists a smooth embedded free boundary minimal hypersurface $\Sigma$ in $M$ such that $\mc{A}_1(M,\partial M)$ is realized by $\Sigma$. Moreover, $\Sigma$ is one of  the following:
\begin{enumerate}
    \item $\Sigma\in\mathcal O$ and has index $\leq 1$. In the case of $\mathrm{ind}(\Sigma)=1$, $\Sigma$ is the min-max minimal hypersurface corresponding to $[M]$ in $H_{n+1}(M,\partial M)$.
    \item $\Sigma\in\mathcal U$ is stable. In this case, $2\area(\Sigma)=\mathcal A_1(M,\partial M)$. Moreover, the 2-sheeted covering is also stable.
\end{enumerate}
\end{theorem}
\begin{proof}
First, we consider the case  $\mathcal{A_S}(M,\partial M)=\mathcal A_1(M,\partial M)$. By Theorem \ref{ques:stable:realize}, either there is $\Sigma\in\mathcal O_\mathcal S$ such that $\area(\Sigma)=\mathcal A_\mathcal S(M,\partial M)$ or  there is $\Sigma\in\mathcal U_\mathcal S$ such that $2\area(\Sigma)=\mathcal A_\mathcal S(M,\partial M)$. If the latter case happens, we claim that the 2-sheeted covering of $\Sigma$ is stable. If not, $\Sigma\in\mathcal U_\mathcal S$ and has unstable 2-sheeted covering. By Proposition \ref{prop:sweepout_nonori} and \ref{prop:non-min-max}, there is a sweepout
\[\Phi:I\longrightarrow \mathcal Z_n(M,\partial M),\]
such that
\begin{enumerate}
  \item $F(\Phi)=\llbracket M\rrbracket$;
  \item $\mathbf L(\Phi)<2\area(\Sigma)$.
\end{enumerate}
By Theorem \ref{thm:min-max}, there is an orientable FBMH $\Sigma'$ with $\area(\Sigma')\leq\mathbf L(\Phi)$ or a non-orientable one $\Sigma''$ with $2\area(\Sigma'')\leq\mathbf L(\Phi)$. Combining the inequalities together,
\[\mathcal A_1(M,\partial M)\leq \area(\Sigma')(2\area(\Sigma''))\leq\mathbf L(\Phi)<2\area(\Sigma)=\mathcal A_1(M,\partial M),\]
which leads to a contradiction.

\vspace{0.5em}
The other case is $\mathcal A_1(M,\partial M)<\mathcal A_\mathcal S(M,\partial M)$. By Proposition \ref{prop:non-orient:less}, it is easy to see that there does not  exist non-orientable FBMH $\Gamma$ such that $2\area(\Gamma)<\mc{A}_\mc{S}(M,\partial M)$. Now we suppose that $\Sigma\in \mc{O}$ is any orientable FBMH such that \[\mc{A}_1(M,\partial M)\leq \area(\Sigma)<\mc{A_S}(M,\partial M).\]
Let $\Phi^\Sigma$ be the sweepout constructed in Proposition \ref{prop:less_area:sweepout}. Denote $\Pi_M$ as the homotopy class corresponding to the fundamental class. By the construction in \ref{prop:sweepout} and Proposition \ref{prop:sweepout_nonori}, it follows that $\Phi^{\Sigma}\in\Pi_M$ and
\begin{equation}\label{eq:from sweepout}
\mathbf L(\Pi_M)\leq\mathcal A_1(M,\partial M).
\end{equation}
Using the min-max theory for compact manifolds with boundary which is developed by Li-Zhou \cite{LZ16}*{Theorem 4.21, Theorem 5.2}, there exists a stationary integral varifold $V$, which is supported on finitely many FBMHs $\Gamma_i$ with multiplicity $n_i$,  such that \[\mf{L}(\Pi_M)=\|V\|(M) =\sum n_i\area(\Gamma_i).\]
Let $\Gamma_i$ be a component of $\Gamma$. If it is non-orientable, then $n_i$ must be even by the arguments in \cite{Zhou15} (see also \cite{Wang}*{Appendix B}). Hence,
\begin{equation}
\mathbf L(\Pi_M)\geq\mathcal A_1(M,\partial M).
\end{equation}
Comparing with (\ref{eq:from sweepout}), we conclude that $\mathrm{\spt}(V)$ is connected and $V=\llbracket \Gamma\rrbracket$ for some $\Gamma\in\mathcal O$ or $V=2\llbracket T\rrbracket$ for some $T\in\mathcal U$.

In the first case, $\Gamma$ is unstable since $\area(\Gamma)=\mathbf{L}(\Pi_M)\leq \mathcal A_1(M,\partial M)<\mathcal A_{\mathcal S}(M,\partial M)$. Obviously, $\area(\Gamma)\geq\mathcal A_1(M,\partial M)$. We conclude that $\area(\Gamma)=\mathcal A_1(M,\partial M)$. Moreover, we claim that $\Gamma$ has index one.

Let $\Phi:I\rightarrow\mathcal Z_n(M,\partial M)$ be the sweepout given by Proposition \ref{prop:less_area:sweepout}. By the process of construction, there exists $\epsilon>0$ such that $\Phi(t)=\llbracket\Gamma_t\rrbracket$ for $-\epsilon<t<\epsilon$, where $\{\Gamma_t\}_{-\epsilon<t<\epsilon}$ forms a foliation around $\Gamma$ constructed in Proposition \ref{prop:barrier:unstable}. That is, $\Gamma_s=F_s(\Gamma)\cap M$, where $\{F_s\}$ is a family of diffeomorphisms of $\widetilde M$ corresponding to $X\in \mathfrak X(M,\Gamma)$. Moreover, $X\big|_{\Gamma}=f\n$ where $f>0$ and satisfies
\begin{equation*}
\left\{\begin{array}{ll}
        Lf>0 \ & \text{in\ } \Sigma\\
        \frac{\partial f}{\partial\eta}<h^{\partial M}(\n,\n)f\ & \text{on\ } \partial\Sigma.
        \end{array}\right.
\end{equation*}

Suppose that $\mathrm{ind}(\Gamma)\geq 2$. Then there exists a function $u$ such that $Q$ (see \ref{equ:index_form}) is negative on the two-dimensional space which is generated by $u$ and $f$. Without loss of generality, $u$ can be chosen to satisfy $Q(f,u)=0$. Now let $Y\in \mathfrak X(M,\Gamma)$ be an extension of $u\n$. Denote $\{G_s\}$ to be a family of diffeomorphisms of $\widetilde M$ generated by $Y$. Take $\theta>0$ small enough, i.e. $\theta\ll\epsilon$. Then for $t,s\in(-\theta,\theta)$, set
\[\Gamma_{t,s}=G_s\left(F_t\left(\Gamma\right)\right).\]

Even if $\Gamma$ has touching set, $\Gamma_{t,s}$ is a smoothly embedded hypersurface in $\widetilde M$ with $\partial\Gamma_{t,s}\subset\partial M$.
By the second variation formula,
\begin{gather*}
  \frac{\partial^2}{\partial t^2}\Big|_{s=t=0}\area(\Gamma_{t,s})=Q(f,f)<0;\\
  \frac{\partial^2}{\partial t\partial s}\Big|_{s=t=0}\area(\Gamma_{t,s})=Q(f,u)=0;\\
  \frac{\partial^2}{\partial s^2}\Big|_{s=t=0}\area(\Gamma_{t,s})=Q(u,u)<0.
\end{gather*}

Hence, there exists $\delta>0$ such that
\begin{equation}\label{equ:less max}
\area(\Gamma_{t,s})<\area(\Gamma)-\delta, \text{\ for\ }|t|+|s|>\theta/2.
\end{equation}

Let $\phi$ be a cut-off function such that
\begin{itemize}
  \item $\phi\geq 0$ everywhere;
  \item $\phi(t)=1$ for $t<\frac{1}{2}$;
  \item $\phi(t)=0$ for $t>1$.
\end{itemize}

Let $\Phi'(t)=\Phi(t)$ if $|t|\geq\theta$ and
\[\Phi'(t)=\llbracket \Gamma_{t,\theta\phi(t/\theta)}\cap M \rrbracket, \text{\ if\ } |t|<\theta.\]

We claim that $\Phi'$ is a sweepout of $(M,\partial M)$. By Almgren's Isomorphism, it suffices to prove that $F(\Phi)=\llbracket M\rrbracket$. Recall that $\Gamma$ generically separates $M$ into $M^-$ and $M^+$. Set $\Psi^-(t):I\rightarrow \mathcal Z_n(M^-,\partial M^-)$ as
\[\Psi^-\left(t\right)=\Phi'\left(\frac{1-t}{2}\right).\]
It follows that $\Psi$ is continuous in the flat topology. Moreover, $\Psi^-$ is a sweepout of $(M^-,\partial M^-,\Gamma)$. Similarly, $\Psi^+:I\rightarrow\mathcal Z_n(M^+,\partial M^+)$ with
\[\Psi^+(t)=-\Phi'\left(\frac{1+t}{2}\right)\]
is a sweepout of $(M^+,\partial M^+,\Sigma)$. By the proof of Proposition \ref{prop:sweepout:B to NB}, we have
\[F(\Psi^-)=-\llbracket M^-\rrbracket, \  F(\Psi^+)=-\llbracket M^+\rrbracket.\]
By the definition of $F$, we conclude that $F(\Phi')=-F(\Psi^-)-F(\Psi^+)=\llbracket M\rrbracket$.

Overall, $\Phi'$ is a sweepout and $F(\Phi')=\llbracket M\rrbracket$. By (\ref{equ:less max}),
\[\sup_{x\in I}\M(\Phi'(x))<\area(\Gamma),\]
which leads to a contradiction to $\mathbf L(\Pi_M)=\area (\Gamma)$. Hence, $\mathrm{ind}(\Gamma)=1$.

In the case of $T$ is non-orientable, $T$ is unstable and $2\area(T)=\mathcal A_1(M,\partial M)$. By Proposition \ref{prop:non-orient:less}, there exists a stable FBMH $T'$ such that $\area(T')<\area(T)$, which implies $\mathcal A_1(M,\partial M)>\mathcal A_{\mathcal S}(M,\partial M)$. This contradicts the assumption that $\mathcal A_1(M,\partial M)<\mathcal A_{\mathcal S}(M,\partial M)$. So we finish the proof.
\end{proof}

\appendix
\section{Isoperimetric Lemmas}
In this section, $(M,\partial M,T)$ is a compact manifold with boundary and portion. The following two lemmas still hold:
\begin{lemma}[$\mathcal F$-isoperimetric lemma, c.f. \cite{LZ16}*{Lemma 3.15}] \label{lemma:F-isoperimetric}
There exists $\epsilon_M>0$ and $C_M>1$ depending only on $M$ such that for any $\tau_1,\tau_2\in\mathcal Z_n(M,\partial M)$ with
\[ \mathcal F(\tau_1-\tau_2)<\epsilon_M,\]
there exists $Q\in\mathbf I_{n+1}(M)$, such that
\begin{itemize}
  \item $\mathrm{spt}(K_1-K_2-\partial Q)\subset\partial M$;
  \item $\mathbf M(Q)<C_M\mathcal F(\tau_1-\tau_2),$
\end{itemize}
where $K_1$ and $K_2$ are the canonical representatives of $\tau_1$ and $\tau_2$.
\end{lemma}

\begin{lemma}[$\mathbf M$-isoperimetric lemma, c.f. \cite{LZ16}*{Lemma 3.17}]\label{lemma:M-isoperimetric}
There exists $\epsilon_M>0$ and $C_M>1$ depending only on $M$ such that for any $\tau_1,\tau_2\in\mathcal Z_n(M,\partial M)$ with
\[\mathbf M(\tau_1-\tau_2)<\epsilon_M,\]
there exists $Q\in\mathbf I_{n+1}(M)$ and $R\in\mathcal R_n(\partial M)$, such that
\begin{itemize}
  \item $K_1-K_2=Q+\partial R$;
  \item $\mathbf M(Q)+\mathbf M(R)<C_M\mathbf M(\tau_1-\tau_2),$
\end{itemize}
where $K_1$ and $K_2$ are the canonical representatives of $\tau_1$ and $\tau_2$.
\end{lemma}

\begin{remark}
In Lemma \ref{lemma:F-isoperimetric}, we emphasize that the support set of $K_1-K_2-\partial Q$ is disjoint from $T$. In Lemma \ref{lemma:M-isoperimetric}, the support set of $R$ is disjoint from $T$.
\end{remark}

\section{Discretization and Interpolation}
In this section, we give a Discretization Theorem and Interpolation Theorem for sweepouts of $(M,\partial M, T)$. We refer to \cite{MN14} and \cite{Zhou15} for closed cases and \citelist{\cite{LZ16}*{\S 4.2}\cite{LMN16}*{\S2.10,2.11}} for compact manifolds with boundary.

\begin{theorem}[Discretization Theorem, \cite{LZ16}*{Theorem 4.12}]\label{thm:discrete}
Given a sweepout of $(M,\partial M,T)$,
\begin{equation*}
\Phi: I\longrightarrow (\mathcal Z_n(M,\partial M),\{\llbracket T\rrbracket\},\{0\}),
\end{equation*}
there exists a sequence of mappings
\begin{equation*}
\phi_i: I(1,j_i)_0\longrightarrow (\mathcal Z_n(M,\partial M),\{\llbracket T\rrbracket\},\{0\})
\end{equation*}
with $j_i<j_{i+1}$ and a sequence of positive numbers $\delta_i\rightarrow 0$ such that
\begin{enumerate}
  \item $S=\{\phi_i\}$ is an $(1,\mathbf M)$-homotopy sequence with $\mathbf M$-fineness
  $\mathbf f_{\mathbf M}(\phi_i)<\delta_i$;
  \item There exists a sequence of $k_i$ such that for all $x\in I(1,j_i)_0$,
  \begin{equation*}
  \mathbf M(\phi_i(x))\leq \sup\{\mathbf M(\Phi(y)):\alpha\in I(1,k_i)_1, x, y\in\alpha\}+\delta_i.
  \end{equation*}
  In particular, we have $\mathbf L(S)\leq\sup_{x\in I}\mathbf M(\Phi(x))$.
  \item $\sup\{\mathcal F(\phi_i(x)-\Phi(x)):x\in I(1,j_i)_0\}<\delta_i$.
\end{enumerate}
\end{theorem}
\begin{proof}
The proof is parallel to the one in \citelist{\cite{MN14}*{Theorem 13.1}\cite{LZ16}*{Theorem 4.12}\cite{Zhou15}*{Theorem 5.5}}. The slightly difference is that we use the $\mathcal F$-isoperimetric lemma \ref{lemma:F-isoperimetric} in $\mathcal Z_n(M,\partial M)$.
\end{proof}

\begin{theorem}[Interpolation Theorem, \cite{LZ16}*{Theorem 4.14}]\label{thm:interpolation}
There exists $C_0>0$ and $\delta_0 >0$ depending only on $(M,\partial M,T)$, such that for every map
\[\psi:I(1,k)_0\rightarrow \mathcal Z_n(M,\partial M)\]
with $\mathbf f_\mathbf M(\psi)<\delta_0$, there exists a continuous map
\[\Psi: I\rightarrow \mathcal Z_n(M,\partial M;\mathbf M)\]
such that
\begin{enumerate}
  \item $\Psi(x)=\psi(x)$ for all $x\in I(1,k)_0$;
  \item for every $\alpha\in I(1,k)_1$, $\Psi|_\alpha$ depends only on the restriction of $\psi$ on the vertices of $\alpha$, and
  \[\sup\{\mathbf M(\Psi(x)-\Psi(y)):x,y\in\alpha\}\leq C_0\mathbf f_\mathbf M(\psi).\]
\end{enumerate}
\end{theorem}
\begin{proof}
The only difference here is that the map needs to satisfy $\mathrm{spt}(\partial \Psi(x))\subset\partial M$. Recall that in the proof of \citelist{\cite{MN14}*{Theorem 14.1}\cite{LZ16}*{Theorem 4.14}}, all the arguments are in the isoperimetric choices. Here, we use Lemma \ref{lemma:M-isoperimetric} and all others are the same.
\end{proof}

\bibliographystyle{plain}

\end{document}

%% file: portion.eps_tex
\begingroup%
  \makeatletter%
  \providecommand\color[2][]{%
    \errmessage{(Inkscape) Color is used for the text in Inkscape, but the package 'color.sty' is not loaded}%
    \renewcommand\color[2][]{}%
  }%
  \providecommand\transparent[1]{%
    \errmessage{(Inkscape) Transparency is used (non-zero) for the text in Inkscape, but the package 'transparent.sty' is not loaded}%
    \renewcommand\transparent[1]{}%
  }%
  \providecommand\rotatebox[2]{#2}%
  \ifx\svgwidth\undefined%
    \setlength{\unitlength}{612bp}%
    \ifx\svgscale\undefined%
      \relax%
    \else%
      \setlength{\unitlength}{\unitlength * \real{\svgscale}}%
    \fi%
  \else%
    \setlength{\unitlength}{\svgwidth}%
  \fi%
  \global\let\svgwidth\undefined%
  \global\let\svgscale\undefined%
  \makeatother%
  \begin{picture}(1,0.70588235)%
    \put(0,0){\includegraphics[width=\unitlength]{portion.eps}}%
    \put(0.38900926,0.14642487){\color[rgb]{0,0,0}\makebox(0,0)[lb]{\smash{portion $T$}}}%
    \put(0.29091401,0.37014168){\color[rgb]{0,0,0}\makebox(0,0)[lb]{\smash{$N$}}}%
    \put(0.7316178,0.54840392){\color[rgb]{0,0,0}\makebox(0,0)[lb]{\smash{$\partial N$}}}%
  \end{picture}%
\endgroup%

%% file: touching.eps_tex
\begingroup%
  \makeatletter%
  \providecommand\color[2][]{%
    \errmessage{(Inkscape) Color is used for the text in Inkscape, but the package 'color.sty' is not loaded}%
    \renewcommand\color[2][]{}%
  }%
  \providecommand\transparent[1]{%
    \errmessage{(Inkscape) Transparency is used (non-zero) for the text in Inkscape, but the package 'transparent.sty' is not loaded}%
    \renewcommand\transparent[1]{}%
  }%
  \providecommand\rotatebox[2]{#2}%
  \ifx\svgwidth\undefined%
    \setlength{\unitlength}{504bp}%
    \ifx\svgscale\undefined%
      \relax%
    \else%
      \setlength{\unitlength}{\unitlength * \real{\svgscale}}%
    \fi%
  \else%
    \setlength{\unitlength}{\svgwidth}%
  \fi%
  \global\let\svgwidth\undefined%
  \global\let\svgscale\undefined%
  \makeatother%
  \begin{picture}(1,0.71428571)%
    \put(0,0){\includegraphics[width=\unitlength]{touching.eps}}%
    \put(0.53784017,0.36670921){\color[rgb]{0,0,0}\makebox(0,0)[lb]{\smash{$\Sigma$}}}%
    \put(0.54421767,0.70174235){\color[rgb]{0,0,0}\makebox(0,0)[lb]{\smash{$\widetilde M$}}}%
    \put(0.48469393,0.03082481){\color[rgb]{0,0,0}\makebox(0,0)[lb]{\smash{$M$}}}%
  \end{picture}%
\endgroup%

%% file: barrier.eps_tex
\begingroup%
  \makeatletter%
  \providecommand\color[2][]{%
    \errmessage{(Inkscape) Color is used for the text in Inkscape, but the package 'color.sty' is not loaded}%
    \renewcommand\color[2][]{}%
  }%
  \providecommand\transparent[1]{%
    \errmessage{(Inkscape) Transparency is used (non-zero) for the text in Inkscape, but the package 'transparent.sty' is not loaded}%
    \renewcommand\transparent[1]{}%
  }%
  \providecommand\rotatebox[2]{#2}%
  \ifx\svgwidth\undefined%
    \setlength{\unitlength}{612bp}%
    \ifx\svgscale\undefined%
      \relax%
    \else%
      \setlength{\unitlength}{\unitlength * \real{\svgscale}}%
    \fi%
  \else%
    \setlength{\unitlength}{\svgwidth}%
  \fi%
  \global\let\svgwidth\undefined%
  \global\let\svgscale\undefined%
  \makeatother%
  \begin{picture}(1,0.70588235)%
    \put(0,0){\includegraphics[width=\unitlength]{barrier.eps}}%
    \put(0.05632019,0.11015149){\color[rgb]{0,0,0}\makebox(0,0)[lb]{\smash{$<\frac{\pi}{2}$}}}%
    \put(0.66503252,0.10730413){\color[rgb]{0,0,0}\makebox(0,0)[lb]{\smash{$H>0$}}}%
  \end{picture}%
\endgroup%

%% file: barrier_stable.eps_tex
\begingroup%
  \makeatletter%
  \providecommand\color[2][]{%
    \errmessage{(Inkscape) Color is used for the text in Inkscape, but the package 'color.sty' is not loaded}%
    \renewcommand\color[2][]{}%
  }%
  \providecommand\transparent[1]{%
    \errmessage{(Inkscape) Transparency is used (non-zero) for the text in Inkscape, but the package 'transparent.sty' is not loaded}%
    \renewcommand\transparent[1]{}%
  }%
  \providecommand\rotatebox[2]{#2}%
  \ifx\svgwidth\undefined%
    \setlength{\unitlength}{612bp}%
    \ifx\svgscale\undefined%
      \relax%
    \else%
      \setlength{\unitlength}{\unitlength * \real{\svgscale}}%
    \fi%
  \else%
    \setlength{\unitlength}{\svgwidth}%
  \fi%
  \global\let\svgwidth\undefined%
  \global\let\svgscale\undefined%
  \makeatother%
  \begin{picture}(1,0.70588235)%
    \put(0,0){\includegraphics[width=\unitlength]{barrier_stable.eps}}%
    \put(0.02971234,0.0900605){\color[rgb]{0,0,0}\makebox(0,0)[lb]{\smash{$<\frac{\pi}{2}$}}}%
    \put(0.59033276,0.06266408){\color[rgb]{0,0,0}\makebox(0,0)[lb]{\smash{$H>0$}}}%
  \end{picture}%
\endgroup%

%% file: generic.eps_tex
\begingroup%
  \makeatletter%
  \providecommand\color[2][]{%
    \errmessage{(Inkscape) Color is used for the text in Inkscape, but the package 'color.sty' is not loaded}%
    \renewcommand\color[2][]{}%
  }%
  \providecommand\transparent[1]{%
    \errmessage{(Inkscape) Transparency is used (non-zero) for the text in Inkscape, but the package 'transparent.sty' is not loaded}%
    \renewcommand\transparent[1]{}%
  }%
  \providecommand\rotatebox[2]{#2}%
  \ifx\svgwidth\undefined%
    \setlength{\unitlength}{612bp}%
    \ifx\svgscale\undefined%
      \relax%
    \else%
      \setlength{\unitlength}{\unitlength * \real{\svgscale}}%
    \fi%
  \else%
    \setlength{\unitlength}{\svgwidth}%
  \fi%
  \global\let\svgwidth\undefined%
  \global\let\svgscale\undefined%
  \makeatother%
  \begin{picture}(1,0.70588235)%
    \put(0,0){\includegraphics[width=\unitlength]{generic.eps}}%
    \put(0.16433927,0.21077043){\color[rgb]{0,0,0}\makebox(0,0)[lb]{\smash{$M$}}}%
    \put(0.18367737,0.40161688){\color[rgb]{0,0,0}\makebox(0,0)[lb]{\smash{$\Sigma$}}}%
    \put(0.82773724,0.21465544){\color[rgb]{0,0,0}\makebox(0,0)[lb]{\smash{$M$}}}%
    \put(0.7541683,0.39313411){\color[rgb]{0,0,0}\makebox(0,0)[lb]{\smash{generic type}}}%
  \end{picture}%
\endgroup%

%% file: separate.eps_tex
\begingroup%
  \makeatletter%
  \providecommand\color[2][]{%
    \errmessage{(Inkscape) Color is used for the text in Inkscape, but the package 'color.sty' is not loaded}%
    \renewcommand\color[2][]{}%
  }%
  \providecommand\transparent[1]{%
    \errmessage{(Inkscape) Transparency is used (non-zero) for the text in Inkscape, but the package 'transparent.sty' is not loaded}%
    \renewcommand\transparent[1]{}%
  }%
  \providecommand\rotatebox[2]{#2}%
  \ifx\svgwidth\undefined%
    \setlength{\unitlength}{504bp}%
    \ifx\svgscale\undefined%
      \relax%
    \else%
      \setlength{\unitlength}{\unitlength * \real{\svgscale}}%
    \fi%
  \else%
    \setlength{\unitlength}{\svgwidth}%
  \fi%
  \global\let\svgwidth\undefined%
  \global\let\svgscale\undefined%
  \makeatother%
  \begin{picture}(1,0.71428571)%
    \put(0,0){\includegraphics[width=\unitlength]{separate.eps}}%
    \put(0.48767013,0.43176108){\color[rgb]{0,0,0}\makebox(0,0)[lb]{\smash{$\Sigma$}}}%
    \put(0.39859694,0.62818881){\color[rgb]{0,0,0}\makebox(0,0)[lb]{\smash{$\widetilde M$}}}%
    \put(0.47682832,0.33652547){\color[rgb]{0,0,0}\makebox(0,0)[lb]{\smash{$M$}}}%
    \put(0.21842798,0.04434026){\color[rgb]{0,0,0}\makebox(0,0)[lb]{\smash{$\Sigma$}}}%
    \put(0.06472893,0.1187453){\color[rgb]{0,0,0}\makebox(0,0)[lb]{\smash{$M^+$}}}%
    \put(0.78184387,0.16657414){\color[rgb]{0,0,0}\makebox(0,0)[lb]{\smash{$\Sigma$}}}%
    \put(0.8003827,0.04039116){\color[rgb]{0,0,0}\makebox(0,0)[lb]{\smash{$M^-$}}}%
  \end{picture}%
\endgroup%